\documentclass[reqno,12pt]{amsart}
\usepackage{mathabx,amssymb,amsmath,amsthm,amsxtra,mathrsfs,calc,bm, color}
\usepackage{enumerate}
\usepackage[margin=1in]{geometry}
\usepackage[]{hyperref}
\usepackage{hypbmsec}
\usepackage{caption}
\usepackage{subcaption}

\usepackage{mathtools}

\usepackage{mleftright}
\mleftright

\usepackage{nccmath}
\usepackage{cases}

\usepackage{calc}
\usepackage{graphicx}
 
\date{\today}

\usepackage{hyperref}

\DeclareRobustCommand{\SkipTocEntry}[5]{}

\let\stdthebibliography\thebibliography
\let\stdendthebibliography\endthebibliography
\renewenvironment*{thebibliography}[1]{%
  \stdthebibliography{OLBC}}
  {\stdendthebibliography}

\def\Z{\mathbb{Z}}

\def\R{\mathbb{R}}
\def\H{\mathbb{H}}
\def\N{\mathbb{N}}
\def\C{\mathbb{C}}

\def\sL{\mathcal L}
\def\sLT{\tilde{\mathcal L}}

\def\sB{\mathcal B}
\def\sQ{\mathcal Q}

\def\sK{\mathcal{K}}

\DeclareMathOperator{\sh}{sh}
\DeclareMathOperator{\ch}{ch}

\DeclareMathOperator{\im}{Im}

\def\SL{{\rm SL}}

\def\hbar{\overline{h}}

\newcommand{\pfrac}[2]{\left(\frac{#1}{#2}\right)}
\newcommand{\pmfrac}[2]{\left(\mfrac{#1}{#2}\right)}
\newcommand{\ptfrac}[2]{\left(\tfrac{#1}{#2}\right)}
\newcommand{\pMatrix}[4]{\left(\begin{matrix}#1 & #2 \\ #3 & #4\end{matrix}\right)}
\newcommand{\ppMatrix}[4]{\left(\!\pMatrix{#1}{#2}{#3}{#4}\!\right)}
\renewcommand{\pmatrix}[4]{\left(\begin{smallmatrix}#1 & #2 \\ #3 & #4\end{smallmatrix}\right)}
\renewcommand{\bar}[1]{\overline{#1}}

\renewcommand{\(}{\left(}
\renewcommand{\)}{\right)}

\renewcommand{\check}{\widecheck}
\renewcommand{\hat}{\widehat}
\renewcommand{\tilde}{\widetilde}

\newcommand{\G}{\mathcal G}

\renewcommand{\sl}{\big|}

\DeclareMathOperator{\sgn}{sgn}

\def\ep{\epsilon}

\newtheorem{theorem}{Theorem}[section]
\newtheorem{lemma}[theorem]{Lemma}
\newtheorem{corollary}[theorem]{Corollary}
\newtheorem{proposition}[theorem]{Proposition}

\theoremstyle{remark}

\numberwithin{equation}{section}

\mathtoolsset{showonlyrefs}

\begin{document}

\author{Scott Ahlgren}
\address{Department of Mathematics\\
University of Illinois\\
Urbana, IL 61801} 
\email{sahlgren@illinois.edu} 

\author{Alexander Dunn}
\address{Department of Mathematics\\
University of Illinois\\
Urbana, IL 61801} 
\email{ajdunn2@illinois.edu} 

\subjclass[2010]{11F30, 11F37, 11L05, 11P82}
\thanks{The first author was  supported by a grant from the Simons Foundation (\#426145 to Scott Ahlgren).}

\dedicatory{\it Dedicated to George Andrews on the occasion of his 80th birthday.}

\title{Maass forms and the  mock theta function $f(q)$}

\begin{abstract}
Let $f(q):=1+\sum_{n=1}^{\infty} \alpha(n)q^n$ be the well-known third order mock theta of Ramanujan.  In 1964, George Andrews proved 
an  asymptotic formula of the form 
\[\alpha(n)= \sum_{c\leq\sqrt{n}} \psi(n)+O_\ep\(n^\ep\),\]
where $\psi(n)$ is an expression involving generalized Kloosterman sums and the $I$-Bessel function.
Andrews conjectured that the series converges to $\alpha(n)$ when extended to infinity, and that it does not converge absolutely.   Bringmann and Ono proved the first of these conjectures.
  Here we  obtain a power savings bound for the error in Andrews' formula, and we also prove the second of these conjectures.
  
  Our methods depend on the spectral theory of Maass forms
of half-integral weight, and in particular on  an average estimate for the Fourier coefficients of such forms which gives a 
 power savings in the spectral parameter.
 
 As a further application of this result, we derive a formula which expresses 
 $\alpha(n)$ with small error as a  sum of exponential terms over imaginary quadratic  points (this is similar in spirit to a recent result of Masri).
 We also obtain a bound 
 for the size of the error term incurred by truncating  Rademacher's analytic formula for the ordinary partition function 
which improves a  result of the first author and Andersen when $24n-23$ is squarefree.  
\end{abstract}
\begin{center}
\dedicatory{\it Dedicated to George Andrews on the occasion of his 80th birthday.}
\end{center}

\maketitle

\section{Introduction}
Let 
\begin{equation}\label{eq:fqdefine}
f(q):=1+\sum_{n=1}^\infty \alpha(n)q^n=1+\sum_{n=1}^\infty \frac{q^{n^2}}{(1+q)^2(1+q^2)^2\cdots(1+q^n)^2}
\end{equation}
be the  famous third-order mock theta function of Ramanujan.  
One  may consult for example \cite{bringmann-ono}, \cite{bfor-book}, \cite{duke-survey},
and \cite{zagier-survey} and the references therein for an account of the substantial body of research related to this 
and to other mock theta functions.
Part of the importance of the function $f(q)$ arises from the fact that the coefficients $\alpha(n)$ 
are related to a fundamental combinatorial statistic.  In particular, we have
\begin{equation*}
\alpha(n)=N_{\rm e}(n)-N_{\rm o}(n),
\end{equation*}
where these denote the number of partitions of even and odd rank respectively.

Ramanujan recorded an asymptotic formula for $\alpha(n)$ in his last letter to Hardy in  1920; this was proved in 1951 by Dragonette \cite{dragonette}.
 Andrews \cite{andrews} made a major breakthrough in his 1964 Ph.D. thesis by proving  the  remarkable formula (valid for any $\ep>0$)
\begin{equation}\label{eq:andrews}
\alpha(n)=\frac\pi{(24n-1)^\frac14}\sum_{c=1}^{\lfloor\sqrt{n}\rfloor}\frac{(-1)^{\lfloor \frac{c+1}{2} \rfloor} A_{2c} \(n-\frac{c(1+(-1)^c)}{4} \)  }{ c}
I_{\frac12}\pfrac{\pi \sqrt{24n-1}}{12c}+O_\ep(n^\ep).
\end{equation}
Here $I_{\frac{1}{2}}$ is the $I$--Bessel function of order $1/2$, and 
  $A_c(n)$ is  the generalized Kloosterman sum
\begin{equation}\label{eq:acdef}
A_c(n):=\sum_{\substack{d \pmod c\\ (d,c)=1}}e^{\pi i s(d,c)}e\(-\frac{dn}c\),\qquad c, n\in \N,
\end{equation}
where $s(d,c)$ is the Dedekind sum defined in \eqref{eq:ded-sum-def} below and $e(x):=\exp(2\pi ix)$.

Andrews   \cite[p. 456]{andrews}, \cite[\S 5]{andrews-survey} conjectured that 
\begin{equation}\label{eq:andrews-infinity}
\alpha(n)=\frac\pi{(24n-1)^\frac14}\sum_{c=1}^{\infty}\frac{(-1)^{\lfloor \frac{c+1}{2} \rfloor} A_{2c} \(n-\frac{c(1+(-1)^c)}{4} \)  }{c}
I_{\frac12}\pfrac{\pi \sqrt{24n-1}}{12c},
\end{equation}
and    that the series does not converge absolutely.
We note that by  a result of Lehmer \cite[Theorem~8]{lehmer-series} we have the Weil-type bound
\begin{equation}\label{eq:lehmer_bound}
\left | A_{2c} \(n-\mfrac{c(1+(-1)^c)}{4} \) \right| \leq 2^{\omega_o(c)} \sqrt{2c},
\end{equation}
where $\omega_o(c)$ is the number of distinct odd prime divisors of $c$
(we give a refinement of this result in Section~\ref{sec:heegnerconvert} below).
This bound does not suffice to  prove convergence.

The formula \eqref{eq:andrews-infinity} was   proved by  Bringmann and Ono \cite{bringmann-ono} in 2006 using the theory of harmonic Maass forms, and in particular the 
 work of  Zwegers  \cite{zwegers} which  packaged Watson's transformation properties  \cite{watson}
for $f(q)$ in a three dimensional vector of real-analytic modular forms.

Here we return to the question of obtaining an effective error estimate for the approximation to $\alpha(n)$ by the truncation 
of the series \eqref{eq:andrews}. To this end, we define the error term $R(n, N)$ by
\begin{equation}\label{eq:Rdef}
\alpha(n)=\frac\pi{(24n-1)^\frac14}\sum_{c\leq N}\frac{(-1)^{\lfloor \frac{c+1}{2} \rfloor} A_{2c} \(n-\frac{c(1+(-1)^c)}{4} \)  }{ c}
I_{\frac12}\pfrac{\pi \sqrt{24n-1}}{12c}+R(n,N).
\end{equation}
Then the result of Andrews gives
\begin{equation}\label{eq:andest}
R(n,\sqrt{n})\ll_\ep n^\ep.
\end{equation}
 Our first main result gives a power-saving improvement.
\begin{theorem}\label{thm:powersave}
Suppose that  $24n-1$ is positive and squarefree. Then for all $\ep>0$ and  $\gamma>0$ we have 
\begin{equation*}
R(n,\gamma \sqrt n)\ll_{\gamma, \ep} n^{-\frac1{147}+\ep}.
\end{equation*}
\end{theorem}

 As a corollary, we see that when $n$ is sufficiently large, $\alpha(n)$ is the 
 closest integer to the truncated sum appearing in \eqref{eq:andrews}.  It would be interesting
 to quantify what ``sufficiently large" means here.

It is also  interesting to note that if one assumes the Ramanujan--Lindel\"{o}f conjecture for the coefficients of Maass cusp forms of weight $1/2$ the present methods would yield 
\begin{equation*}
R(n,\gamma \sqrt{n}) \ll_{\gamma,\ep} n^{-\frac{1}{16}+\ep},   
\end{equation*}
while the analogue of the Linnik--Selberg conjecture for the sums of generalized Kloosterman sums which arise in the proof would give 
\begin{equation*}
R(n,\gamma \sqrt{n}) \ll_{\gamma,\ep} n^{-\frac{1}{4}+\ep}.
\end{equation*}

We prove the second conjecture of Andrews mentioned above 
using a character sum identity proved in Section~\ref{sec:heegnerconvert} together with an equidistribution result of Duke, Friedlander and Iwaniec \cite{duke-fried-iwan-cong}   for  solutions of quadratic congruences to prime moduli.
\begin{theorem}\label{thm:abs_conv_not}
The series \eqref{eq:andrews-infinity} does not converge absolutely for any value of $n$.
\end{theorem}

For  squarefree values of  $24n-1$,
Masri \cite[Theorem 1.3]{masri-ranks}   obtained an asymptotic   formula of the form 
\begin{equation*}
\alpha(n)=M(n)+O_\ep\(n^{-\frac1{240}+\ep}\),
\end{equation*}
where $M(n)$ is a twisted  finite sum of terms $\exp(2\pi \im \tau)$ as $\tau$ ranges over distinguished Galois orbits of  Heegner points on $X_0(6)$. This relies on a general power saving bound for traces of modular functions over such orbits, as well as 
a result of  Alfes \cite{alfes} which relates the values $\alpha(n)$ to traces of certain real-analytic modular functions.

As a consequence of  Theorem~\ref{thm:powersave} and  the results  in Section~\ref{sec:heegnerconvert}, 
we   obtain an asymptotic formula for $\alpha(n)$ as a sum over a set of quadratic points in the upper half-plane $\H$.
Suppose that  $D>0$, and
define
\[
	\sQ_{-D, 12} := \left\{ ax^2+bxy+cy^2: b^2-4ac=-D, \ 12\mid a, \ a>0 \right\}.
\] 
Then $\Gamma_0(12)$ acts on this set from the left and preserves     $b\pmod{12}$.
If $Q=[12a, b, c]\in \sQ_{-D, 12}$, 
 define $\chi_{-12}(Q)=\pfrac{-12}{b}$,
and let  $\tau_Q$ denote the root of $Q(\tau, 1)$ in $\H$, so that $g\tau_Q=\tau_{gQ}$ for $g\in \Gamma_0(12)$. From this discussion the summands
in the next theorem are well-defined.

\begin{theorem}\label{thm:alg}
Suppose that $24n-1$ is positive and squarefree.  Then for all $\ep>0$ and  $\gamma>0$  we have 
\[
\alpha(n)=\frac{i}{\sqrt{24n-1}}\sum_{\substack{Q\in \Gamma_\infty\backslash \sQ_{1-24n, 12}  \\ \im \tau_Q>\gamma}}
\chi_{-12}(Q)\(e(\tau_Q)-e(\overline\tau_Q)\)+O_{\gamma, \ep}\(n^{-\frac1{147}+\ep}\).\]
\end{theorem}

Our methods depend on  bounds for sums of Kloosterman sums attached to a half-integral weight multiplier on $\Gamma_0(2)$
 which are uniform with respect to all parameters. These depend in turn on the spectral theory of Maass forms, and in particular on 
 average bounds for the coefficients $\rho(n)$ of such forms which are uniform both in the argument $n$ and with respect to the Laplace eigenvalue $\lambda$.
 We rely on three such bounds to treat various ranges of the parameters involved. The first two are recent results of Andersen-Duke and of the first author with Andersen.
  The third is a new bound which is important in obtaining the exponent $-\frac1{147}$ which appears in the theorems above. 
 This  result (which holds in any level) gives a significant improvement in $\lambda$ aspect at the cost of a small loss in $n$ aspect; it is recorded  as  Theorem~\ref{trade} below.

 The series \eqref{eq:andrews-infinity} is reminiscent of Rademacher's well-known series
  \cite{rademacher-partition-function,rademacher-partition-series} for the ordinary partition function $p(n)$:
 \begin{equation} \label{rad}
p(n)=\frac{2 \pi}{(24n-1)^{\frac{3}{4}}} \sum_{c=1}^{\infty} \frac{A_c(n)}{c} I_{\frac{3}{2}}  \pfrac{\pi \sqrt{24n-1}}{6c}
\end{equation}
(this does converge absolutely, in contrast with \eqref{eq:andrews-infinity}).
A classical problem is to estimate the error associated with truncating this series, and the methods of this paper 
give an improvement for this estimate.  In analogy with \eqref{eq:Rdef}, define $S(n, N)$ by 
\begin{equation*}
p(n):=\frac{2 \pi}{(24n-1)^{\frac{3}{4}}}
 \sum_{c\leq N} \frac{A_c(n)}{c} I_{\frac{3}{2}}  \pfrac{\pi \sqrt{24n-1}}{6c} +S(n, N).
\end{equation*}
Rademacher \cite{rademacher-partition-function,rademacher-partition-series} proved that 
$S(n,\gamma \sqrt{n}) \ll_\gamma n^{-\frac{1}{4}}$, and Lehmer \cite{lehmer-series} improved this to 
 $S(n,\gamma \sqrt{n}) \ll_\gamma n^{-\frac12}\log n$. 
When $24n-23$ is squarefree, Folsom and Masri \cite{folsom-masri} proved that 
\begin{equation*}
S \left(n, \sqrt{\mfrac{n}{6}} \right) \ll n^{-\frac{1}{2}-\delta} \quad \text{for some} \quad \delta>0.
\end{equation*} 
Recently the first author and Andersen \cite{ahlgren-andersen} obtained the bound
\begin{equation*} \label{AAremainder}
S \left(n, \gamma \sqrt{n} \right) \ll_{\gamma, \ep} n^{-\frac{1}{2}-\frac{1}{168}+\ep}.
\end{equation*}
As another application of Theorem~\ref{trade}, we obtain the following.
 \begin{theorem}\label{thm:pofn}
Suppose  that $24n-23$ is positive and squarefree. Then for all $\ep>0$ and $\gamma>0$ we have
\begin{equation*} 
S(n,\gamma \sqrt n)\ll_{\gamma,\ep} n^{-\frac12-\frac1{147}+\ep}.
\end{equation*}
\end{theorem}

 We close with a  brief outline of the contents of the paper.  In the next section, we develop some background material on 
 the spectral theory of automorphic forms  and Kloosterman sums.  In Section~\ref{sec:fq} we develop the properties
 of a particular multiplier which is related to the coefficients of $f(q)$.
 
To prove  Theorem~\ref{thm:powersave} requires  bounds for sums of Kloosterman sums which are uniform in all parameters.
The analysis in Section 7 is similar to that of \cite{ahlgren-andersen}
(and is similar in spirit to the work of Sarnak and Tsimerman \cite{sarnak-tsimerman} in weight $0$ on $\SL_2(\Z)$, although significant
complications arise from the multiplier of weight $\frac12$).
The Kuznetsov trace formula is the basic tool to relate sums of Kloosterman sums to the coefficients of Maass forms. 
 Section~\ref{Kuznetsov1} contains a version of the Kuznetsov formula in the mixed sign case for half integral weight multipliers,
 Section~\ref{sec:endgame} contains the analysis which proves Theorem~\ref{thm:powersave}, and   Section~\ref{sec:pofn} contains a  sketch
of the proof of Theorem~\ref{thm:pofn}.

 In Section~\ref{sec:ell2} we state  three  estimates for coefficients of Maass forms which are crucial for our work.
 The first is a mean value estimate which was recently proved by  Andersen and Duke \cite{andersen-duke}.
 The second, which  was proved in \cite{ahlgren-andersen}, is an average version of a well-known result of Duke \cite{duke-half-integral}.

 The third is the average estimate Theorem~\ref{trade} mentioned above;  the proof of this result is quite involved, and occupies Section~\ref{sec:main}.
 We make crucial use of a  version of the Kuznetsov trace formula which appears in a recent paper of Duke, Friedlander and Iwaniec 
 \cite{duke-fried-iwan2017}.  We follow the basic method of Duke \cite{duke-half-integral} but with a modified test function which
 leads to a savings with respect to the spectral parameter.
   Duke's method relies in turn on estimates of Iwaniec \cite{iwaniec-fourier-coefficients}
 for averages of Kloosterman sums in level aspect.  One of the terms arising from the Kusnetsov formula is an infinite sum over  weights $\ell$ of spaces of holomorphic cusp forms, and much of  the technical difficulty arises from the need to bound the summands uniformly in terms of $\ell$.

 In Section~\ref{sec:heegnerconvert} we prove the  key identity Proposition~\ref{lem:fkmk} which expresses the Kloosterman sums arising in \eqref{eq:andrews}
 as Weyl-type sums, and we use this identity to prove Theorem~\ref{thm:alg}.
Finally, in the last section we prove Theorem~\ref{thm:abs_conv_not}.
 
In the body of the paper we make the following convention: in equations which involve an arbitrary small positive quantity $\ep$, 
the constants which are implied by the notation $\ll$ or $O$ are allowed to depend on $\ep$.  Any other dependencies in the implied constants
will be explicitly noted.

\section*{Acknowledgments}
We thank Nick Andersen and Wadim Zudilin for their helpful comments.  We also thank the referee for comments which improved our
exposition.

\section{Background}\label{sec:back}
We begin with some brief background material on Maass  forms with general weight and multiplier. 
 For more details one may consult \cite[Section 2]{ahlgren-andersen}, 
\cite[Section 2]{duke-fried-iwan2017} (where it is assumed  that the cusp $\infty$ is singular),
  or \cite{sarnak-additive}.

Let $k$ be a real number and let $\H$ denote the upper half-plane.
For $\gamma=\pMatrix abcd\in \SL_2(\R)$  and $\tau=x+iy\in \H$, we define
\[
 j(\gamma,\tau) := \frac{c\tau+d}{|c\tau+d|} = e^{i\arg(c\tau+d)}\]
 and the weight $k$ slash operator by 
 \[f\sl_k \gamma := j(\gamma,\tau)^{-k} f(\gamma \tau),\]
 where we choose the argument in $(-\pi,\pi]$.
 
For each $k$, the  Laplacian
\begin{equation}\label{eq:delta_def}
	\Delta_k := y^2 \bigg( \frac{\partial^2}{\partial x^2} + \frac{\partial^2}{\partial y^2} \bigg) - iky \frac{\partial}{\partial x}
\end{equation}
 commutes with the weight $k$ slash operator.
 
For simplicity we will work only with the groups  $\Gamma_0(N)$ for $N\in \N$ and with weights $k\in \frac12\Z$,
although much of what is said here holds in more generality.
 Let $\Gamma$ denote such a group.
 We say that $\nu: \Gamma \rightarrow \C^{\times}$ is a multiplier system of weight $k$ if 
\begin{itemize}
\item $|\nu|=1$
\item $\nu(-I)=e^{-\pi i k}$, and 
\item $\nu(\gamma_1 \gamma_2) j(\gamma_1 \gamma_2,\tau)^k=\nu(\gamma_1) \nu(\gamma_2) j(\gamma_2,\tau)^k j(\gamma_1, \gamma_2 \tau)^k$ for all $\gamma_1,\gamma_2 \in \Gamma$.
\end{itemize}

Given a cusp $\mathfrak{a}$, let $\Gamma_{\mathfrak{a}}:=\{\gamma \in \Gamma: \gamma \mathfrak{a}=\mathfrak{a} \}$ denote the  
stabilizer in $\Gamma$ and let $\sigma_{\mathfrak{a}}$ denote the unique (up to translation on the right) matrix in 
$\SL_2(\mathbb{R})$ satisfying $\sigma_{\mathfrak{a}} \infty=\mathfrak{a}$ and
 $\sigma_{\mathfrak{a}}^{-1} \Gamma_{\mathfrak{a}} \sigma_{\mathfrak{a}}=\Gamma_{\infty}$. 
 Define $\alpha_{\nu,\mathfrak{a}} \in [0,1)$ by the condition 
\begin{equation*}
\nu \( \sigma_{\mathfrak{a}} \pMatrix 1101
 \sigma_{\mathfrak{a}}^{-1} \)=e \({-\alpha_{\nu,\mathfrak{a}}} \).
\end{equation*}
The cusp $\mathfrak{a}$ is singular with respect to  $\nu$ if $\alpha_{\nu,\mathfrak{a}}=0$.
When  $\mathfrak a=\infty$ we suppress the subscript.

If $\nu$ is multiplier of weight $k$, then it is a multiplier in any weight $k'\equiv k\mod 2$, and 
 $\bar\nu$ is a multiplier of weight $-k$. 
If $\alpha_{\nu}=0$ then  $\alpha_{\bar\nu}=0$, while 
 if $\alpha_{\nu}>0$ then $\alpha_{\bar\nu}=1-\alpha_{\nu}$.
For $n\in \Z$ we define
\[n_\nu:=n-\alpha_\nu;\]
then we have 
\begin{equation}\label{eq:n_nu_conj}
n_{\bar\nu}=\begin{cases} 
			-(1-n)_\nu\quad&\text{if $\alpha_\nu\neq 0$},\\
			n\quad&\text{if $\alpha_\nu= 0$}.
		\end{cases}
\end{equation}

With this notation we  define the generalized Kloosterman sum (at the cusp $\infty$) by
\begin{equation}\label{eq:kloos_def}
	S(m,n,c,\nu) := \sum_{\substack{0\leq a,d<c \\ \gamma=\pmatrix abcd\in \Gamma}} \bar\nu(\gamma) e\pfrac{m_\nu a+n_\nu d}{c}.
\end{equation}
We have the relationships
\begin{equation}\label{eq:nuconj}
\overline{S(m,n,c,\nu)}=
	\begin{cases}
	S(1-m, 1-n, c, \overline\nu)&\quad\text{if $\alpha_\nu>0$,}\\
	S(-m, -n, c, \overline\nu)&\quad \text{if $\alpha_\nu=0$.}
	\end{cases}
\end{equation}	

Two important multipliers of weight $\frac12$ are 
 the  eta-multiplier $\nu_\eta$ on $\SL_2(\Z)$, given by 
\begin{equation}\label{eq:etamult}
\eta(\gamma\tau)=\nu_\eta(\gamma)\sqrt{c\tau+d}\,\eta(\tau), \qquad \gamma=\pMatrix abcd\in \SL_2(\Z),
\end{equation}
and the theta-multiplier $\nu_\theta$ on $\Gamma_0(4)$, given by 
\begin{equation}\label{eq:thetamult}
\theta(\gamma\tau)=\nu_\theta(\gamma)\sqrt{c\tau+d}\,\theta(\tau), \qquad \gamma=\pMatrix abcd\in \Gamma_0(4).
\end{equation}
Here $\eta(\tau)$ and $\theta(\tau)$ are the two fundamental theta functions
\[\begin{aligned}
\eta(\tau)&:=q^\frac1{24}\prod_{n=1}^\infty(1-q^n),\\
\theta(\tau)&:=\sum_{n=-\infty}^\infty q^{n^2},
\end{aligned}
\]
where we use the standard notation
\[q:=e(\tau)=e^{2\pi i\tau}.\]

For $\nu_\theta$ we have the formula
\begin{equation} \label{eq:def-theta-mult}
	\nu_\theta \ppMatrix abcd = \pfrac cd \ep_d^{-1},
\end{equation}
where $\ptfrac\bullet\bullet$ is the extended  Kronecker symbol and 
\[
	\ep_d =
	\begin{cases}
		1 & \text{ if }d\equiv 1\pmod 4, \\
		i & \text{ if }d\equiv 3\pmod 4.
	\end{cases}
\]
From this we obtain
\begin{equation}\label{eq:thetaconj}
\bar{\nu_\theta}(\gamma)=\pmfrac{-1}d\nu_\theta(\gamma), \qquad \gamma=\pMatrix abcd\in \Gamma_0(4).
\end{equation}

With $d\mu:=\frac{dx\, dy}{y^2}$, 
define
\[\lVert f \rVert^2=\int_{\Gamma_0(N)\backslash \H}|f|^2\, d\mu.\]
Denote by $\sL_k(N, \nu)$ the space of $L^2$ functions which satisfy
\begin{equation}\label{eq:trans}
f(\gamma \tau)=j(\gamma, \tau)^k\nu(\gamma)\,f(\tau)\qquad \text{for all $\gamma\in \Gamma_0(N)$.}
\end{equation}
Let $\sB_k(N, \nu)$ denote the subspace of $\sL_k(N, \nu)$ consisting of smooth functions $f$ such that $f$ and $\Delta_k f$ are bounded on $\H$.
Then  $\Delta_k$  has a unique self-adjoint extension to $\sL_k(N, \nu)$, which we also denote by $\Delta_k$.
For each singular cusp $\mathfrak a$ (and only at such cusps) 
there is an Eisenstein series $E_{\mathfrak a}(z, s)$.  These provide the continuous spectrum, which covers  $[1/4, \infty)$.

The reminder of the spectrum is discrete and of finite multiplicity. We denote the discrete spectrum by 
\[\lambda_0\leq \lambda_1\leq\dots, \]
where we have
\[\lambda_0\geq \frac{|k|}2\(1-\frac{|k|}2\).\]
One component of the discrete spectrum is provided by residues of the Eisenstein series 
$E_{\mathfrak a}(z,s)$ at possible simple poles $s$ with $\frac12<s\leq 1$;
the corresponding  eigenvalues  have $\lambda<\frac14$. 
 The remainder of the discrete spectrum
arises from   Maass cusp forms.

Denote by $\sLT_k(N, \nu)$ the subspace of $\sL_k(N, \nu)$ spanned by eigenfunctions of $\Delta_k$.
If $f\in \sLT_k(N, \nu)$ has Laplace eigenvalue $\lambda$, then 
we write
\[\lambda=\frac14+r^2,\qquad r\in  i\,(0, i/4]\cup[0, \infty),\]
and refer to $r$ as the spectral parameter of $f$.  
Denote by  $\sLT_k(N, \nu, r)$  the subspace of such  functions. 
Let $W_{\kappa, \mu}$ denote the usual $W$-Whittaker function.
Then each  $f\in \sLT_k(N, \nu, r)$  
has a Fourier expansion of the form
\begin{equation}\label{eq:f_fourier}
f(\tau)=c_{0}(y) +  \sum_{n_\nu\neq 0} \rho(n) W_{\frac{k\sgn(n_\nu)}2, ir}(4\pi |n_\nu| y)e(n_\nu x),
\end{equation}
where
\[
c_{0}(y)=\begin{cases} 0\quad&\text{if $\alpha_\nu\neq 0$},\\
				      0\quad&\text{if $\alpha_\nu=0$ and $r\geq 0$,}\\
				      \rho(0)y^{\frac12+i r}&\text{if $\alpha_\nu=0$ and $r\in i(0, 1/4]$,}
		\end{cases}
\]				      	
with coefficients $\rho(n)$ (see \cite[p. 3878]{proskurin-new} or \cite[p. 2509]{duke-fried-iwan2017})  .  Note that in the last case, we have $\rho(0)\neq 0$ only when $f$ arises as a residue.

Complex conjugation gives an isometry (of normed spaces)
\[\sLT_k(N, \nu, r)\longleftrightarrow \sLT_{-k}(N, \bar\nu, r).\]
If $f\in \sLT_k(N, \nu, r)$, then using \eqref{eq:f_fourier}   with \eqref{eq:n_nu_conj} and 
the fact that $W_{\kappa, \mu}\in \R$ when $\kappa\in \R$ and  $\mu\in \R\cup i\R$
\cite[(13.4.4), (13.14.3), (13.14.31)]{nist},
we find that the coefficients $\rho_c(n)$ of $f_c:=\bar f$ satisfy
\begin{equation}\label{eq:acconj}
\rho_c(n)=\begin{cases}
	\bar{\rho(1-n)}&\quad\text{if $\alpha_\nu>0$ and $n\neq 0$,}\\
	\bar{\rho(-n)}&\quad \text{if $\alpha_\nu=0$.}
	\end{cases}
\end{equation}	

The Maass lowering operator
\begin{equation*}
	L_k := \frac k2 + 2iy \frac{\partial}{\partial \bar \tau} = \frac k2 + iy \left(\frac{\partial}{\partial x} + i\frac{\partial}{\partial y}\right)
\end{equation*}
gives a map 
\[
\sLT_k(N, \nu,r)\longrightarrow \sLT_{k-2}(N, \nu,r)
\]
and satisfies
\begin{equation}\label{eq:L-k-norm}
\lVert L_k f \rVert^2 = \(r^2 + \mfrac{(k-1)^2}{4}\) \lVert f \rVert^2
=\(\lambda-\mfrac k2\(1-\mfrac k2\)\)\lVert f \rVert^2.
\end{equation}

From the last equation, we see that if  $f\in \sLT_k(N, \nu)$ has the minimal eigenvalue $\frac{|k|}2\(1-\frac{|k|}2\)$,  then $f(\tau)$ is in the kernel of 
$L_k$ if $k\geq 0$, and $\bar f(\tau)$ is in the kernel of $L_{-k}$ if $k<0$.
Let $M_k(N, \nu)$ denote the space of holomorphic modular forms of weight $k$ and multiplier $\nu$ on $\Gamma_0(N)$.
Using \eqref{eq:trans}, it follows that the function
\[
F(\tau):=\begin{cases} y^{-\frac k2} f(\tau)\quad&\text{if $k\geq 0$},\\
				   y^{\frac k2} \bar{f}(\tau)\quad&\text{if $k< 0$}\\	
		\end{cases}
\]
lies in  $M_k(N, \nu)$  if $k\geq 0$ and in $M_{-k}(N, \bar\nu)$ if $k<0$.
Also, $F(\tau)$ is a cusp form if and only if $f(\tau)$ is a Maass cusp form.

Suppose that $f\in \sLT_k(N, \nu, r)$.  Then using \eqref{eq:f_fourier} and
\cite[(2.16)]{ahlgren-andersen}, we have the expansion
\begin{equation}\label{eq:Lk_fourier}
L_k f(\tau)=c^{(L)}_{0}(y) +  \sum_{n_\nu\neq 0} \rho^{(L)}(n) 
W_{\frac{(k-2)\sgn(n_\nu)}2, ir}(4\pi |n_\nu| y)e(n_\nu x),
\end{equation}
where
\begin{equation}\label{eq:Lk_coeff}
\rho^{(L)}(n)=\begin{cases}
	-\(r^2 + \mfrac{(k-1)^2}{4}\)\rho(n)&\quad\text{if $n_\nu>0$,}\\
	\rho(n)&\quad \text{if $n_\nu<0$,}
	\end{cases}
\end{equation}
and 
\[
c^{(L)}_{0}(y)=\(\mfrac{k-1}2-ir\) c_0(y).
\]

\section{A multiplier on $\Gamma_0(2)$ and a  formula for the coefficients of $f(q)$}\label{sec:fq}
Here we relate the coefficients $\alpha(n)$ to Kloosterman sums attached to a multiplier $\psi$ on $\Gamma_0(2)$.
For the eta-multiplier defined in the last section,
 we have a formula of Rademacher \cite[(74.11), (74.12)]{rademacher-book} which is valid for $c>0$: 
\begin{equation} \label{eq:chi-dedekind-sum}
	\nu_\eta(\gamma) = \sqrt{-i} \, e^{-\pi i s(d,c)} \, e\pfrac{a+d}{24c},
\end{equation}
where $s(d,c)$ is the Dedekind sum
\begin{equation}  \label{eq:ded-sum-def}
	s(d,c) = \sum_{r=1}^{c-1} \mfrac rc \, \left(\mfrac{dr}{c} - \left\lfloor\! \mfrac{dr}{c}\!\right\rfloor - \mfrac 12\right).	
\end{equation}
For  $c>0$  and $\gamma=\pmatrix abcd$ we have another convenient formula \cite[\S 4.1]{knopp}
\begin{equation} \label{eq:chi-kronecker-symbol}
	\nu_\eta(\gamma) = 
	\begin{dcases}
		\( \mfrac dc \) \, e\(\mfrac 1{24} \left[(a+d)c-bd(c^2-1)-3c\right]\) & \text{ if $c$ is odd}, \\
		\(\mfrac cd \) \, e\(\mfrac 1{24} \left[(a+d)c-bd(c^2-1)+3d-3-3cd\right]\) & \text{ if $c$ is even.}
	\end{dcases}
\end{equation}
We have 
$\nu_\eta\(\pm\pmatrix 1b01\)=e\pfrac b{24}$.  Finally, 
if $c>0$ we   have $\nu_\eta(-\gamma)=i \nu_\eta(\gamma)$ (this follows since $\gamma$ and $-\gamma$ act the same way on $\H$).

For $\gamma=\pMatrix abcd\in \Gamma_0(2)$ define 
\begin{equation}\label{eq:psidef}
\psi \(\gamma\)= \begin{cases}
i^ {c/2}\pfrac{-1}d \overline{\nu_\eta}(\gamma) & \text{if} \quad c\equiv 0\pmod 4, \\
i^ {c/2} \overline{\nu_\eta}(\gamma) & \text{if} \quad c\equiv 2\pmod 4.
 \end{cases}
\end{equation}
One can compute using  \eqref{eq:chi-dedekind-sum} to see that the real-analytic 
form $\tilde M$  appearing on page 251 of  \cite{bringmann-ono} 
satisfies 
\[\tilde M(\gamma\tau)=\psi(\gamma)(c\tau+d)^\frac12\tilde M(\tau)\qquad\text{for all $\gamma=\pMatrix abcd\in \Gamma_0(2)$.}
\]
(this can also be derived from  \cite[Theorem~2.2]{andrews}).
By  \cite[Proposition 2.1]{hejhal-stf2} it follows that $\psi$ is a multiplier of weight $\frac12$ on $\Gamma_0(2)$.

For the cusp $\infty$, we have
$ \psi\(\pmatrix 1101 \)= e \({-\frac1{24} }\),$
so that $\alpha_\psi=\frac1{24}$.
For the cusp $0$ we may take $\sigma_0=\pmatrix0 {-1/\sqrt{2}} {\sqrt{2}} 0$.
Then the formulas above give 
 \begin{equation*}
\psi \(\sigma_0 \pMatrix 1101\sigma_0^{-1}\)=\psi \(\pMatrix 10{-2}1\)=e\(-\mfrac{1}{3}\).
\end{equation*}
\begin{lemma}\label{lem:fq-kloos}  Let $A_c(n)$ and $\psi$ be defined as in \eqref{eq:acdef}, \eqref{eq:psidef}.  Then for $c>0$ we have 
\[
(-1)^{\lfloor \frac{c+1}{2} \rfloor} A_{2c} \(n-\frac{c(1+(-1)^c)}{4} \)=e\pmfrac18\overline{S(0,n,2c,\psi)}.
\]
\end{lemma}
\begin{proof}
This follows from   a case-by-case computation using   \eqref{eq:chi-dedekind-sum}
together with the fact that $s(-d,c)=-s(d,c)$.
\end{proof}

Since $\alpha(n)\in \Z$, the formula \eqref{eq:andrews-infinity} becomes 
\begin{equation}\label{eq:alphafinal}
\alpha(n)=\frac{2 \pi} {(24n-1)^{\frac{1}{4}}}e \(-\mfrac18\) \!\!\!\!\sum_{\substack{c>0  \\  c \equiv 0  \pmod{2}}}  \frac{S(0,n,c,\psi)}{c} I_{\frac{1}{2}} \( \frac{\pi \sqrt{24n-1}}{6c} \).
\end{equation}

We will work in the space $\sLT_\frac12(2, \psi)$. 
By the discussion above, neither cusp is singular, so there are no Eisenstein series for this multiplier.
\begin{lemma}\label{lem:level-up}
For each $r$, the map $\tau\mapsto24\tau$ gives an injection
\begin{equation*}\label{eq:24tau}
\sLT_\frac12(2, \psi, r)\longrightarrow
\sLT_\frac12\(144, \pmfrac{12}\bullet\nu_\theta,r\).
\end{equation*}
\end{lemma} 
\begin{proof} It is  enough to check the transformation law.
 Given  $f\in \sLT_\frac12(2, \psi, r)$, define
\[g(\tau):=f(24\tau)=f\big|_\frac12\pmatrix{\sqrt{24}}00{1/\sqrt{24}}.\]
If $\gamma=\pmatrix abcd\in \Gamma_0(144)$ with $c>0$, then we have
\[g\big|_\frac12\gamma=f\big|_\frac12\gamma'\,\pmatrix{\sqrt{24}}00{1/\sqrt{24}},\]
where $\gamma'=\pmatrix a{24b}{c/24}d$.
Then a  case by case computation using \eqref{eq:psidef} and 
\eqref{eq:chi-kronecker-symbol} shows that 
\[\psi(\gamma')=\pmfrac{12}d\nu_\theta(\gamma).\]
The identities $e\pfrac{1-d}8=\pfrac2d\ep_d$ and $\pfrac{-1}d\ep_d=\ep_d^{-1}$ for odd $d$ are useful 
for this computation. 
\end{proof}

\section{Three estimates for the coefficients of Maass forms} \label{sec:ell2}
The proofs of our main results will depend on three different 
average estimates for the Fourier coefficients of Maass forms.

The first is a restatement of a recent result of Andersen and Duke \cite[Theorem 4.1]{andersen-duke}
 (we state this only in the case  of the cusp $\infty$).
 Suppose that $\nu$ is a multiplier of weight $\frac12$ for $\Gamma_0(N)$ and that for $n_\nu\neq 0$  the bound
  \begin{equation}\label{eq:adassump}
  \sum_{\substack{c>0\\c\equiv 0\pmod N}}\frac{|S(n,n,c,\nu)|}{c^{1+\beta}}\ll_\nu |n_\nu|^\ep
  \end{equation}
  holds for some $\beta\in (\frac12, 1)$.

The result of  \cite{andersen-duke} is stated for positive $n$ in weights $\pm\frac12$.  To derive the statement below 
 we use \eqref{eq:n_nu_conj},  \eqref{eq:nuconj}, and  \eqref{eq:acconj}.  Note that the assumption \eqref{eq:adassump} differs slightly from 
that of  \cite{andersen-duke} in that $n_\nu$ appears in place of $n$ on the right side; an examination of the proof 
shows that this is sufficient.  This allows us to access the case when $n=0$ and $n_\nu<0$, which is important in our applications.

Whenever we speak of an orthonormal basis $\{v_j(\tau)\}$ for $\sLT_\frac12(N, \nu)$, we  assume that each $v_j$ is an eigenform of $\Delta_\frac12$ with eigenvalue
$\lambda_j$ and spectral parameter $r_j=\sqrt{\lambda_j-1/4}$.

\begin{proposition}[Andersen-Duke]  \label{AndDu} Suppose that   $\nu$ is a multiplier on $\Gamma_0(N)$ of weight $\frac12$ which satisfies \eqref{eq:adassump}. Let $b_j(n)$ denote the coefficients of an
 orthonormal basis $\{v_j(\tau)\}$ for $\sLT_\frac12(N, \nu)$.
Then we have
\[
|n_\nu|\sum_{x\leq r_j\leq 2x}|b_j(n)|^2e^{-\pi r_j}\ll_{\ep, N}\(x^2+|n|^{\beta+\ep} x^{1-2\beta}\log^\beta x\)\cdot
\begin{cases}
x^{-\frac12}\quad&\text{if $n_\nu>0$,}\\
x^{\frac12}\quad&\text{if $n_\nu <0$.}
\end{cases}
\]
\end{proposition}

The second estimate is a restatement of \cite[Proposition 8.2]{ahlgren-andersen}.

\begin{proposition}  \label{avgduke}
Let $D$ be an even fundamental discriminant and let $N$ be a positive integer with $D\mid N$.  
Let $b_j(n)$ denote the coefficients of an
 orthonormal basis $\{v_j(\tau)\}$ for $\sLT_\frac12 \(N,\pfrac{|D|}\bullet \nu_{\theta}\)$.
 Then for  square-free $n\neq 0$ and  $x \geq 1$  we have 
\begin{equation*}
|n| \sum_{0\leq r_j \leq x} \frac{|b_j(n)|^2}{\ch\pi r_j} \ll_{\ep,N} |n|^{\frac{3}{7}+\ep} x^{5-\frac{\sgn n}{2}}.
\end{equation*}
\end{proposition}
This follows directly from \cite[Proposition 8.2]{ahlgren-andersen} for $n>0$. 
If  $n<0$, then using  \eqref{eq:L-k-norm},
we see that the map
\[v_j\mapsto \(r_j^2+\mfrac1{16}\)^{-\frac12}\bar{L_\frac12v_j}
\]
gives an isometry between the subspaces of $\sLT_\frac12 \(N,\pfrac{|D|}\bullet \nu_{\theta}\)$
and $\sLT_\frac32 \(N,\pfrac{|D|}\bullet \bar\nu_{\theta}\)$
spanned by those forms with spectral parameter not equal to $i/4$. Moreover, if we  denote the coefficients of $\(r_j^2+\frac1{16}\)^{-\frac12}\bar{L_\frac12v_j}$ by $a_j(n)$, then \eqref{eq:acconj} and  \eqref{eq:Lk_coeff} give
\[a_j(-n)=\(r_j^2+\frac1{16}\)^{-\frac12}\,\bar{b_j(n)}\qquad\text{for $n<0$.}
\]
The statement follows by applying the result of \cite{ahlgren-andersen} to the forms in weight $\frac32$
and using partial summation.

Finally, we will prove an estimate which is slightly weaker in $n$ aspect than Proposition~\ref{avgduke} but is significantly better in spectral parameter aspect.  

\begin{theorem} \label{trade}
Let $D$ be an even fundamental discriminant and let $N$ be a positive integer with $D\mid N$.  
Let $b_j(n)$ denote the coefficients of an
 orthonormal basis $\{v_j(\tau)\}$ for $\sLT_\frac12 \(N,\pfrac{|D|}\bullet \nu_{\theta}\)$.
Then for  square-free $n\neq 0$ and  $x \geq 1$  we have 
\begin{equation*}
|n| \sum_{ |r_j| \leq x} \frac{|b_j(n)|^2}{\ch \pi r_j} \ll_{\ep,N} |n|^{\frac{131}{294}+\ep} x^{3-\frac{\sgn n}2}.
\end{equation*}
\end{theorem}

The proof of this result is somewhat involved, and occupies the next section. We follow the basic strategy of Duke \cite{duke-half-integral},
which in turn relies on bounds of Iwaniec \cite{iwaniec-fourier-coefficients} for sums of Kloosterman sums averaged over the level.
We use a   recent version of the Kusnetsov trace formula 
due to Duke, Friedlander and Iwaniec \cite{duke-fried-iwan2017} which allows us to use test functions which lead to a significant savings 
with respect to the spectral parameter. The most delicate part of the subsequent analysis involves a sum over 
holomorphic cusp forms of arbitrarily large weight $\ell$ weighted by $J$-Bessel transforms, and much of the difficulty arises 
in obtaining bounds which are uniform with respect to $\ell$. 
\subsection{Remarks}
For individual coefficients, this theorem implies that 
\begin{equation}\label{eq:individual}
b_j(n)\ll_{N, \ep} \lambda_j^{\frac34-\frac {\sgn n} 8}\ch \pmfrac{\pi r_j}2 |n|^{-\frac{163}{588}+\ep}.
\end{equation}
This is slightly weaker in $n$-aspect and significantly stronger in $\lambda$-aspect than the well-known result of Duke \cite[Theorem~5]{duke-half-integral}.
Baruch and Mao \cite[Theorem 1.5]{baruch-mao} obtained a bound which is stronger in $n$ than Duke's bound, but weaker in the spectral parameter.
The approach of Baruch and Mao relies on subconvexity bounds for $L$-functions due to Blomer, Harcos, and Michel \cite{blomer-harcos-michel}
and a Kohnen-Zagier type formula which relates special values of these $L$-functions to the coefficients $b_j(n)$.  The subconvexity bounds
were later improved by Blomer and Harcos \cite{blomer-harcos}.  Following the approach of Baruch and Mao using these improved bounds should lead
to a bound for the individual coefficients in which the exponent $-\frac{163}{588}=-.277\dots$ above is replaced by $-\frac{5}{16}=-.3125$, with the same power of $\lambda_j$.
An average version of this bound in the style of Theorem~\ref{trade} would improve the exponents in our applications.

A very strong average version is available in special cases due to work of Young \cite{young}. 
Young obtains estimates of the form
\[\sum_{T\leq r_j\leq T+1}L(u_j\times \chi_q, \tfrac12)^3\ll_\ep(q(T+1))^{1+\ep},\]
where $q$ is odd and square-free, $\chi_q$ is a real character of conductor $q$, 
and $\{u_j\}$ is an orthonormal basis of weight zero Maass cusp forms in level dividing $q$
(note that there are $\asymp T$ terms in the sum). 
Unfortunately this does not apply to our situation.
Andersen and Duke \cite{andersen-duke} use Young's result as an input to obtain strong bounds for sums of Kloosterman sums arising in the plus-space in level four
(in which case the forms $u_j$ have level one), and obtain striking applications.

Finally, we mention that the Lindel\"of hypothesis $L(u_j\times \chi_q,\tfrac12)\ll_\ep((1+|r_j|)q)^\ep$ would lead to the pointwise bound
\begin{equation*}
b_j(n)\ll_{N, \ep} \lambda_j^{-\frac {\sgn n} 8}\ch \pmfrac{\pi r_j}2 |n|^{-\frac12+\ep}.
\end{equation*}
This would of course lead to substantial improvements in the error bounds discussed in this paper, as described in the Introduction.

\section{Proof of Theorem~\ref{trade}} \label{sec:main}

\subsection{The Kusnetsov formula}
We follow the exposition of Duke, Friedlander and Iwaniec \cite[\S 2]{duke-fried-iwan2017}.
Assume that $\nu$ is a multiplier on $\Gamma_0(N)$, and that 
\[\alpha_\nu=0.\]
Let $\Phi:[0,\infty) \rightarrow \R$ be a smooth  function such that 
\begin{equation} \label{smoothcond}
\Phi(0)=\Phi^{\prime}(0)=0
\end{equation}
and let $J_s$ denote the $J$-Bessel function.
 For $s \in \mathbb{C}$, define
\begin{equation*}
\tilde{\Phi}(s)=\int_{0}^{\infty} J_s(u) \Phi(u)\,  \frac{du}u,
\end{equation*}
and suppose  that for some $\ep>0$ we have
\begin{align}
\tilde{\Phi}(2it)  &\ll t^{-2-\ep} \ch \pi t \qquad \text{for  $ t \geq 1$}, \label{DFI1}\\
\tilde{\Phi}(\ell) &\ll \ell^{-2-\ep} \qquad\quad \, \ \ \ \text{for $\ell\geq 1$}.\label{DFI2}
\end{align}
Also define
\begin{equation*}
\hat{\Phi}(t)=i \( \tilde{\Phi}(2it) \cos  \pi \(\mfrac k2+it \) -\tilde{\Phi}(-2it) \cos  \pi \(\mfrac k2-it \)  \) \frac{D_k(t)}{\sh \pi t},
\end{equation*}
where 
\begin{equation*}
D_k(t)=\frac{1}{2 \pi^2} \Gamma \( \mfrac{1+k}2+it \) \Gamma \(\mfrac{1+k}2-it  \).
\end{equation*}

For $m, n\geq 1$  define
\begin{equation}\label{eq:kphidef}
\sK^{(N)}_{\Phi}(m,n):=\sum_{\substack {c>0 \\ c \equiv 0 \pmod{N}}} c^{-1} S(m,n,c,\nu) \Phi  \pfrac{4 \pi \sqrt{mn}}{c} .
\end{equation}
The Kusnetsov formula expresses $\sK^{(N)}_{\Phi}(m,n)$ as the sum of three spectral terms.
The first two of these correspond to the discrete and the continuous spectrums.
Let $\{u_j(\tau)\}$ be an orthonormal basis  for $\sLT_k(N, \nu)$,
and denote the coefficients   by  $\rho_j(n)$ and the  spectral parameters  by $r_j$.
For each singular cusp $\mathfrak a$, let $\rho_{\mathfrak{a}}(n,t)$ be the Fourier coefficients for the Eisenstein series attached  to  $\mathfrak{a}$.  Define 
\begin{align}
\mathcal{L}^{(N)}_{\hat{\Phi}}(m,n)&:=4 \pi \sqrt{mn} \sum_{j \geq 0}  \overline{\rho_j(m)} \rho_j(n) \frac{\hat{\Phi}(r_j)}{\ch \pi r_j}, \label{eq:calLdef} \\
\mathcal{M}^{(N)}_{\hat{\Phi}}(m,n)&:=4 \pi \sqrt{mn} \sum_{\mathfrak{a}} \frac{1}{4 \pi} \int_{-\infty}^{\infty} \overline{\rho_{\mathfrak{a}}(m,t)} \rho_{\mathfrak{a}}(n,t) \frac{\hat{\Phi}(t)}{\ch \pi t} dt.\label{eq:calMdef}
\end{align}

The third term involves the holomorphic forms. For $\ell\equiv k\pmod 2$
 with $\ell\geq 2$,  let $S_\ell(N,\nu)$ denote the space of holomorphic cusp
 forms of weight $\ell$, level $N$ and multiplier $\nu$. Let $\mathcal{B}_\ell$ denote an orthonormal basis of this space with respect to the Petersson inner product 
\begin{equation*}
\langle f,g \rangle =\int_{\Gamma_0(N) \setminus \mathbb{H}} y^\ell f(\tau) \bar{g(\tau)}\, d \mu.
\end{equation*}
Every $f \in \mathcal{B}_\ell$ has a Fourier expansion 
\begin{equation}\label{eq:cuspfourier}
f(\tau)=\sum_{n=1}^{\infty} a_f(n) (4 \pi n)^\frac{\ell-1}2 e(n \tau).
\end{equation}
The third spectral term is  
\begin{equation*}
\mathcal{N}^{(N)}_{\check{\Phi}}(m,n)=\sum_{\substack{\ell \equiv k \!\!\!\! \pmod{2} \\ \ell \geq 2}} \check{\Phi}(\ell) \Gamma(\ell) 
\sum_{f \in \mathcal{B}_\ell} \bar{a_f(m)} a_f(n),
\end{equation*}
where 
\begin{equation}\label{eq:phicheckdef}
\check{\Phi}(\ell):=\pi^{-1} \tilde{\Phi}(\ell-1) e^{\frac{\pi i(\ell-k)}{2}}.
\end{equation}

Define 
\begin{equation*}
\gamma_k:=e^{-\frac{\pi i k }{2}}.
\end{equation*}
Duke, Friedlander, and Iwaniec \cite[Theorem~2.5, (2.28), (2.29)]{duke-fried-iwan2017} 
proved the following.
\begin{proposition}[Duke--Friedlander--Iwaniec] \label{DFI}
Suppose that $\nu$ is a multiplier of weight 
$k$ on $\Gamma_0(N)$ with $\alpha_\nu=0$.
Let  $\Phi$ be a smooth function which satisfies \eqref{smoothcond}, \eqref{DFI1} and \eqref{DFI2}. Then for $m,n \geq 1$  we have 
\begin{equation*}
\gamma_k \sK^{(N)}_{\Phi}(m,n)=\mathcal{L}^{(N)}_{\hat{\Phi}}(m,n)+\mathcal{M}^{(N)}_{\hat{\Phi}}(m,n)+\mathcal{N}^{(N)}_{\check{\Phi}}(m,n).
\end{equation*}
\end{proposition}

\subsection{Start of the proof of Theorem~\ref{trade}}\label{sec:start}
We now suppose that $D$ is an even fundamental discriminant, that $N\equiv0\pmod 8$ is a positive integer with $D\mid N$, and that 
\begin{equation}\label{eq:weightchar}
(k, \nu)=\(\mfrac12, \pmfrac{|D|}\bullet \nu_{\theta}\)\qquad \text{or}
 \qquad(k, \nu)=\(-\mfrac12, \pmfrac{|D|}\bullet \bar{\nu_{\theta}}\)=\(-\mfrac12, \pmfrac{-|D|}\bullet \nu_{\theta}\).
\end{equation}
For such a multiplier $\nu$, we have the Weil bound   \cite[Lemma 4]{waibel}
\begin{equation}\label{eq:weilbound}
\left|S(n, n, c, \nu)\right|\leq\tau(c) (n, c)^\frac12 c^{\frac12}.
\end{equation}
Suppose that  $\mu\in\{-1, 0, 1\}$.  For  $n\in \N$ and $x \geq 1$ we   define 
\begin{equation}\label{eq:kloospm1}
 K _{\mu}^{(N)}(n;x):=\sum_{\substack{\\ c \leq x \\ c \equiv 0 \pmod{N}  }} c^{-\frac{1}{2}} S(n,n,c,\nu) e \( \frac{2\mu n}{c} \).
\end{equation}

Let $P$ be a positive parameter (which will eventually be set to $n^\frac17$), and define
\begin{equation} \label{eq:Qdef}
\sQ=\sQ(n, N, P):=\left\{pN:  p\  \text{prime,}\ \  P<p \leq 2P, \text{ and}\ \  p \nmid 2nN  \right\}.
\end{equation}
For the theta-multiplier, Iwaniec \cite[Theorem~3]{iwaniec-fourier-coefficients} obtained a bound for the sums $K _{\mu}^{(N)}(n;x)$ averaged over the level.  
This was extended to the case
of twists by a quadratic character by Waibel~\cite[\S 3]{waibel} (this is also implicit in the work of Duke~\cite{duke-half-integral}). Combining these results, we have
\begin{proposition}\label{prop:Iwanbound}
Suppose that $N \equiv 0 \pmod{8}$, that $\mu\in\{-1, 0, 1\}$, that $n>0$ is square-free, and that $\sQ$ is as in \eqref{eq:Qdef}. Suppose that $\nu$ is one of the characters appearing in \eqref{eq:weightchar}. Then we have 
\begin{equation} \label{iwaniec-fourier-coefficients}
\sum_{Q \in \sQ} \left|  K_\mu^{(Q)}(n;x) \right |  \ll_{N, \ep}  \left\{xP^{-\frac{1}{2}}+x n^{-\frac{1}{2}}+(x+n)^{\frac{5}{8}} \( x^{\frac{1}{4}} P^{\frac{3}{8}}+n^{\frac{1}{8}} x^{\frac{1}{8}} P^{\frac{1}{4}} \)\right\} (nx)^{\ep}.
\end{equation}
\end{proposition}

We choose the smooth function 
\begin{equation} \label{phichoice}
\Phi(u):=\mfrac18\sqrt{\mfrac\pi2}\, u^{-\frac{1}{2}} J_{\frac{9}{2}}(u).
\end{equation}
At $u=0$ we have the expansion $\Phi(u)=\frac{u^4}{7560}+\cdots$, so  
\eqref{smoothcond} is satisfied.
Using the Weber--Schafheitlin integral \cite[(10.22.57)]{nist} we compute
\begin{equation} \label{t1}
\tilde{\Phi}(2it)=-  i   \frac{t(1+t^2) } {(1+4t^2)(9+4t^2)(25+4t^2)}\ch \pi t  \quad \text{for} \quad t \in \mathbb{R},
\end{equation}
and
\begin{equation} \label{t2}
\tilde{\Phi}(\ell)=-\frac18 \frac{\ell(\ell^2-4) \cos  \pfrac{ \pi\ell}{2}}{(\ell^2-1)(\ell^2-9)(\ell^2-25)}  \quad \text{for} \quad \ell \geq 1.
\end{equation}
Thus both \eqref{DFI1} and \eqref{DFI2} are satisfied (note that the cosine factor cancels the zeros in the denominator of \eqref{t2}).

For $k=\pm\frac12$ we have  
\begin{equation*}
\Hat{\Phi}(t)=   \frac{\sqrt 2\, t(1+t^2)} {(1+4t^2)(9+4t^2)(25+4t^2)}\,(\ch \pi t)^2\frac{D_k(t)}{\sh(\pi t)}.
\end{equation*}
We have  $\hat{\Phi}(t)>0$ for   $t \in\R\cup i(0, 1/4]$. 
When  $k=-\frac12$, the factor $D_k(t)$ produces a pole of $\hat\Phi$ at $t=i/4$. 
However, this value does not occur in the sum  \eqref{eq:calLdef}, 
since this value of the spectral parameter $r_j$ corresponds to the minimal eigenvalue $\lambda_j=\frac3{16}$, which does not arise
 by the  discussion at the end of Section~\ref{sec:back}  since there are no holomorphic
modular forms in negative weight.  In this  case  every eigenvalue has $\lambda_j\geq \frac14-\pfrac7{32}^2$;
this follows from the discussion in \cite[\S 3]{sarnak-additive} and the lower bound $\frac14-\pfrac{7}{64}^2$ which is available in weight 
zero by the work of  Kim and Sarnak \cite[Appendix 2]{kim-sarnak}.

\begin{figure}[h]
\centering
\begin{subfigure}[t]{0.48\textwidth}
	\includegraphics[width=\textwidth,height=1.8in]{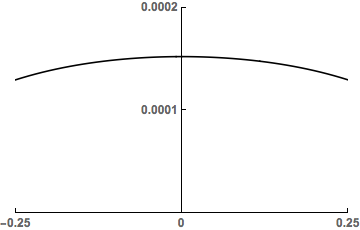}
    \caption{$\hat\Phi(i t)$ for $t\in [-1/4, 1/4]$ when $k=1/2$.}
\end{subfigure}
\hfill
\begin{subfigure}[t]{0.48\textwidth}
	\includegraphics[width=\textwidth,height=1.8in]{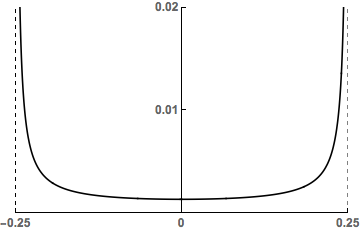}
    \caption{$\hat\Phi(i t)$ for $t\in (-1/4, 1/4)$ when $k=-1/2$.}
    \label{kfigB}
\end{subfigure}
\caption{Plots of $\hat\Phi(t)$ for imaginary $t$.}
\label{plots}	
\end{figure}

Let  $n\in \N$ and let $\sQ$ be as in \eqref{eq:Qdef}.
For  each $Q\in \sQ$,  Proposition~\ref{DFI} gives
\begin{equation*}
\mathcal{L}^{(Q)}_{\hat{\Phi}}(n,n)+\mathcal{M}^{(Q)}_{\hat{\Phi}}(n,n)=\gamma_k \sK^{(Q)}_{\Phi}(n,n)-\mathcal{N}^{(Q)}_{\check{\Phi}}(n,n).
\end{equation*}
From \eqref{eq:calLdef} and \eqref{eq:calMdef} we see that  $\mathcal{L}^{(Q)}_{\hat{\Phi}}(n,n)$ and $\mathcal{M}^{(Q)}_{\hat{\Phi}}(n,n)$ are positive, from which 
\begin{equation}\label{eq:keyineq0} 
\mathcal{L}^{(Q)}_{\hat{\Phi}}(n,n) \leq   \gamma_k \sK^{(Q)}_{\Phi}(n,n)-\mathcal{N}^{(Q)}_{\check{\Phi}}(n,n).
\end{equation}

 For each $Q \in \sQ$, the functions $\left\{[\Gamma_0(N):\Gamma_0(Q)]^{-\frac{1}{2}} u_j \right\}$ form an orthonormal subset of $\sLT_k(Q,\nu)$.
Since $[\Gamma_0(N):\Gamma_0(Q)] \leq p+1 \ll P$,
we find that 
\begin{equation*}
\mathcal{L}^{(Q)}_{\Hat{\Phi}}(n,n)  \geq \frac{\mathcal{L}^{(N)}_{\Hat{\Phi}}(n,n) }{[\Gamma_0(N):\Gamma_0(Q)]} \gg  \frac{\mathcal{L}^{(N)}_{\Hat{\Phi}}(n,n) }{P}.
\end{equation*}
Since  $|\sQ| \asymp P/ \log P$,  summing \eqref{eq:keyineq0} over $Q$ gives
\begin{equation}  \label{keyineq2}
 \frac{1}{\log P} \mathcal{L}^{(N)}_{\Hat{\Phi}}(n,n)\ll   \sum_{Q \in \sQ} \left | \sK^{(Q)}_{\Phi}(n,n) \right |+\sum_{Q \in \sQ}\left | \mathcal{N}^{(Q)}_{\check{\Phi}}(n,n) \right |.
\end{equation}
As in \cite{iwaniec-fourier-coefficients} and \cite{duke-half-integral}, we 
choose 
\[P=n^{\frac{1}{7}}.\]

In the next two sections we  bound the right-hand side of \eqref{keyineq2}. 

\subsection{Treatment of $\mathcal{N}^{(Q)}_{\check{\Phi}}(n,n)$}
When $k=\pm\frac12$ we have 
\begin{equation*}
\mathcal{N}^{(Q)}_{\check{\Phi}}(n,n)=\sum_{\substack{\ell \equiv k \pmod{2} \\ \ell > 2}} \check{\Phi}(\ell) \Gamma(\ell)
 \sum_{f \in \mathcal{B}_\ell}  |a_f(n)|^2.
\end{equation*}
Recalling the normalization \eqref{eq:cuspfourier}, we have Petersson's formula (c.f. \cite[Proposition 2.3]{duke-fried-iwan2017})
\begin{equation*}
\sum_{f \in \mathcal{B}_\ell}  |a_f(n)|^2=\frac{1}{\Gamma(\ell-1)} \( 1 + 2 \pi i^{-\ell} 
\sum_{ c \equiv 0 \pmod{Q}} c^{-1} S(n,n,c,\nu)  J_{\ell-1} \( \mfrac{4 \pi n}{c} \) \).
\end{equation*}
Recalling the definition \eqref{eq:phicheckdef}, we obtain 
\begin{multline} \label{keyn}
\sum_{Q \in \sQ}\left|\mathcal{N}^{(Q)}_{\check{\Phi}}(n,n)\right|
\ll P \sum_{\substack{\ell \equiv k \pmod{2} \\ \ell > 2  }}\ell\, | \tilde{\Phi}(\ell) | \\ 
+
 \sum_{\substack{\ell \equiv k \pmod{2} \\ \ell > 2  }}\ell\, | \tilde{\Phi}(\ell) | \sum_{Q \in \sQ} \Bigg| \sum_{ c \equiv 0 \pmod{Q}} c^{-1} S(n,n,c,\nu)  J_{\ell-1} \( \mfrac{4 \pi n}{c} \) \Bigg|.
\end{multline}
The first sum on the right is $\ll P$ by \eqref{t2}.
It is important to obtain bounds for the inner term which are uniform in $\ell$ as well as $n$.
For small $\ell$ we are able to control the dependence on $\ell$ explicitly, and for large $\ell$ we exploit the decay of  $\tilde{\Phi}(\ell)$.
Let $\beta>0$ be a parameter to be chosen later.  In the next two subsections we treat the ranges $\ell\leq n^\beta$ and $\ell>n^\beta$
separately.

\subsubsection{Small values of $\ell$: $\ell\leq n^{\beta}$}
For these values of $\ell$, we treat the three ranges
\begin{equation*}
1 \leq c \leq n/\ell^2, \quad n/\ell^2 \leq c \leq n \quad \text{and} \quad c\geq n.
\end{equation*}

In the first range we use an explicit representation for $J_{\ell-1}(z)$ when $\ell\in\frac12\N-\N$.
For such $\ell$ and for $k\geq 0$ define
\begin{equation*}
c_{\ell, k}:=\frac{\Gamma(\ell-\frac{1}{2}+k)}{k! \Gamma(\ell-\frac{1}{2}-k)},  
\end{equation*}
and define the polynomial 
\begin{equation*}
H_\ell(z):=\frac{e^{-\frac{\pi i(2\ell-1)}4}}{2\pi}\sum_{k=0}^{\ell-\frac32}i^k c_{\ell, k}z^k.
\end{equation*}
Using the first formula in  \cite[7.11]{bateman}, a computation shows that   for   $z\in \R$ and  $\ell\in\frac12\N-\N$ we have
\begin{equation}\label{eq:nicelommel}
 \sqrt z J_{\ell-1}(2\pi z)=e(z)H_\ell\pmfrac1{4\pi z}+e(-z)\bar H_\ell\pmfrac1{4\pi  z}.
\end{equation}

Recall the definition \eqref{eq:kloospm1}.
Using \eqref{eq:nicelommel} and 
partial summation, we obtain 
\begin{equation}\label{partsum}
\begin{aligned} 
 &\sum_{\substack{1 \leq c \leq n/\ell^2 \\  c \equiv 0 \pmod{Q}}} c^{-1} S(n,n,c,\nu)  J_{\ell-1} \pfrac{4 \pi n}{c}  \\
  &= \frac{n^{-\frac12}}{\sqrt 2} \(H_\ell\pfrac1{8\pi\ell^2}K_1^{(Q)}\(n; \frac{n}{\ell^2}\) +
  \bar H_\ell\pfrac1{8\pi\ell^2}K_{-1}^{(Q)}\(n; \frac{n}{\ell^2}\)\)\\
  &\qquad\qquad -  \frac{n^{-\frac{3}{2}}}{8\pi\sqrt 2} \int_1^{n/\ell^2} H'_\ell\pfrac{x}{8 \pi n} K_1^{(Q)}(n; x)\, dx  
  -\frac{n^{-\frac{3}{2}}}{8\pi\sqrt 2} \int_1^{n/\ell^2}  \bar H'_\ell\pfrac{x}{8 \pi n} K_{-1}^{(Q)}(n;x)\, dx.
\end{aligned}
\end{equation}

For $0 \leq k \leq \ell-\frac32$ we have 
\begin{equation*}
c_{\ell, k}\leq\frac{(2\ell)^{2k}}{k!}.
\end{equation*}
Therefore  $H_\ell \( \frac{1}{8 \pi \ell^2} \) \ll 1$ (the series converges as $\ell\rightarrow\infty$), and 
we have

\begin{equation*}
H'_\ell  \pfrac{x}{8 \pi n}\ll\ell^2\quad \text{for } \quad 1 \leq x \leq n/\ell^2.
\end{equation*}
Using  \eqref{partsum} and  Proposition~\ref{prop:Iwanbound}  with $P=n^\frac17$ we obtain
\begin{equation} \label{est1}
\sum_{Q \in \sQ}\Bigg|  \sum_{\substack{1 \leq c \leq n/\ell^2 \\  c \equiv 0 \pmod{Q}}} c^{-1} S(n,n,c,\nu)  J_{\ell-1} \pmfrac{4 \pi n}{c} \Bigg|
 \ll \(\ell^{-\frac12}n^{\frac{3}{7}} +\ell^{-\frac14} n^{\frac{23}{56}} \)(\ell n)^{\ep}.
\end{equation}

Now consider the range $n/\ell^2 \leq c \leq n$. 
We have \cite[(10.6.1)]{nist}
\begin{equation}\label{eq:jderiv}
2J_{\ell-1}'(z)=J_{\ell-2}(z)-J_\ell(z).
\end{equation}
We also have  a uniform bound for the $J$-Bessel function due to  L. Landau \cite{landau}.
\begin{proposition}[Landau] \label{uniformJ}
For $v>0$ and   $x>0$,  we have
\begin{equation} \label{uJ}
\left|J_v(x) \right|\leq c_0 x^{-\frac13},
\end{equation}
where $c_0=0.7857\dots$.
\end{proposition}

Using Proposition~\ref{uniformJ} we find that 
\begin{equation} \label{derivbounds}
\(x^{-\frac12}J_{\ell-1}\pmfrac{4\pi n}x\)'
\ll  n^{-\frac13}x^{-\frac76}+n^\frac23 x^{-\frac{13}6 }.
\end{equation}
Partial summation  using this bound
together with Proposition~\ref{prop:Iwanbound} and \eqref{derivbounds}  and a careful examination of the error terms which arise yields
\begin{equation} \label{est2}
 \sum_{Q \in \sQ} \Bigg| \sum_{\substack{n/\ell^2 \leq c \leq n \\  c \equiv 0 \pmod{Q}}} 
  c^{-1} S(n,n,c,\nu)  J_{\ell-1} \pmfrac{4 \pi n}{c} \Bigg| 
 \ll \(\ell^{\frac{11}{6}} n^{\frac{3}{7}}+\ell^{\frac{25}{12}} n^{\frac{23}{56}}\)(\ell n)^\ep.
\end{equation}

For the remaining range $c\geq n$, the treatment of Iwaniec \cite[pp. 400-401]{iwaniec-fourier-coefficients} applies
uniformly in $\ell$.  To see this, note that \eqref{eq:jderiv} and the  inequality  $|J_{\ell-1}(x)|\leq \frac{x^{\ell-1}}{\Gamma(\ell-1)}$ for $x>0$ 
\cite[(10.14.4)]{nist} give the estimate
\begin{equation} \label{simplerderiv}
\(x^{-\frac12}J_{\ell-1}\pmfrac{4\pi n}x\)' \ll n x^{-\frac{5}{2}} \quad \text{for } \quad x\geq n
\end{equation}
which is used in that argument.
We conclude that for all $\ell$ we have
\begin{equation} \label{est3}
\sum_{Q \in \sQ} \Bigg| \sum_{\substack{c\geq n \\  c \equiv 0 \pmod{Q}}} c^{-1} S(n,n,c,\nu) J_{\ell-1} \pmfrac{4 \pi n}{c}\Bigg|
 \ll n^{\frac{3}{7}+\ep}.
\end{equation}
From  \eqref{t2},   \eqref{est1}, \eqref{est2} and \eqref{est3} we obtain 
\begin{multline} \label{smallL}
\sum_{\substack{\ell \equiv k \pmod{2} \\   2 < \ell \leq  n^{\beta} }}\ell\, | \tilde{\Phi}(\ell) | \sum_{Q \in \sQ}
\Bigg| \sum_{c \equiv 0 \pmod{Q}} c^{-1} S(n,n,c,\nu)  J_{\ell-1} \( \mfrac{4 \pi n}{c} \) \Bigg| \\ 
 \ll \sum_{2 <\ell\leq n^{\beta}} \ell^{-2} 
 \(\ell^{\frac{11}{6}} n^{\frac{3}{7}}+\ell^{\frac{25}{12}} n^{\frac{23}{56}}\)(\ell n)^{\ep} \ll n^{\frac{3}{7}+\frac56 \beta+\ep}+n^{\frac{23}{56}+\frac{13}{12} \beta+\ep}.
\end{multline}

\subsubsection{Large values of $\ell$: $\ell>n^{\beta}$}
Let $0<\gamma<1$ be a parameter to be chosen later. We consider the  three ranges
\begin{equation*}
1 \leq c \leq n^{\gamma}, \quad n^{\gamma} \leq c \leq n, \quad \text{and} \quad c\geq n.
\end{equation*} 
We begin with a simple lemma.
\begin{lemma} \label{trivial}
Let $\sQ$ be as in \eqref{eq:Qdef}. For $b>-1$ we have
\begin{equation*}
\sum_{Q \in \sQ} \sum_{\substack {c\leq x\\ c \equiv 0 \pmod{Q}}} c^b (n,c)^{\frac{1}{2}} \ll_\ep x^{b+1} n^\ep\log P.
\end{equation*}
\end{lemma}
\begin{proof}
The inner sum is 
\[\leq \sum_{d\mid n}d^\frac12\!\!\!\!\!\!\sum_{\substack{c \leq x \\ c \equiv 0 \pmod{[d,Q]}}} c^{b}
\ll x^{b+1}\sum_{d\mid n}\frac{d^\frac12}{[d,Q]}
\leq	x^{b+1} \frac{(n, N)}{Q}\sum_{d\mid n}\frac1{d^\frac12}\ll x^{b+1}\frac{ N}{Q}n^\ep,
\]
where we have used the fact that such  $Q$ have $Q=pN$ with $p\nmid n$.
Writing $Q=pN$ and summing over $p$ gives the lemma.
\end{proof}

In the first range, we estimate using the  Weil bound \eqref{eq:weilbound},  Proposition~\ref{uniformJ}, and Lemma~\ref{trivial}. We obtain
\begin{equation} \label{initrang}
\sum_{Q \in \sQ} \Bigg| \sum_{\substack{1 \leq c \leq n^{\gamma} \\  c \equiv 0 \pmod{Q}}} c^{-1} S(n,n,c,\nu)J_{\ell-1}  \pmfrac{4 \pi n}{c}\Bigg|
   \ll \sum_{Q \in \sQ} n^{-\frac{1}{3}} \sum_{\substack{1 \leq c \leq n^{\gamma} \\  c \equiv 0 \pmod{Q}}} c^{-\frac{1}{6}} (n,c)^{\frac{1}{2}} \ll n^{\frac{5}{6} \gamma-\frac{1}{3}+\ep}.
\end{equation}
In the second range we use partial summation together with  Proposition~\ref{prop:Iwanbound} and  \eqref{derivbounds}.
Examining the error terms which arise from this computation, we  obtain  
\begin{multline} \label{midrang}
\sum_{Q \in \sQ} \Bigg| \sum_{\substack{n^\gamma \leq c \leq n \\  c \equiv 0 \pmod{Q}}} c^{-1} S(n,n,c,\nu)J_{\ell-1}  \pmfrac{4 \pi n}{c} \Bigg|
  \\
     \ll \(n^{\frac{3}{7}}
     +n^{\frac{61}{42}-\frac{25}{24}\gamma}
     +n^{\frac{113}{84}-\frac{11}{12}\gamma}
      +n^{\frac{25}{42}-\frac{1}{6}\gamma}
      +n^{\frac{19}{42}-\frac{1}{24}\gamma}     
      +n^{\frac{29}{84}+\frac{1}{12}\gamma}
      +n^{-\frac{17}{42}+\frac{5}{6}\gamma}\)n^\ep.            
\end{multline}
For the third range we recall that \eqref{est3} holds for all $\ell$.

We choose $\gamma=\frac{47}{49}$ to balance \eqref{initrang} and \eqref{midrang}
(this is the value for which $\frac{113}{84}-\frac{11}{12}\gamma=\frac{5}{6} \gamma-\frac{1}{3}$).
For this range of $\ell$ we obtain
\begin{equation} \label{total}
\sum_{Q \in \sQ}  \Bigg| \sum_{\substack{n^\gamma \leq c \leq n \\  c \equiv 0 \pmod{Q}}} c^{-1} S(n,n,c,\nu)J_{\ell-1}  \pmfrac{4 \pi n}{c} \Bigg|
  \ll n^{\frac{137}{294}+\ep},
\end{equation} 
from which 
\begin{multline} \label{largeL}
\sum_{\substack{\ell \equiv k \pmod{2} \\   \ell> n^{\beta} }}\ell |  \tilde{\Phi}(\ell) | \sum_{Q \in \sQ}
\Bigg|\sum_{c \equiv 0 \pmod{Q}} c^{-1} S(n,n,c,\nu)  J_{\ell-1} \( \mfrac{4 \pi n}{c} \) \Bigg|  \\ \ll
 n^{\frac{137}{294}+\ep}\sum_{\ell>n^\beta} \ell^{-2}  
  \ll n^{\frac{137}{294}-\beta+\ep}.
\end{multline}
Finally, we choose $\beta=\frac{1}{49}$ to balance \eqref{smallL} and \eqref{largeL}.  From  \eqref{keyn}, we obtain
\begin{equation} \label{Nest}
\sum_{Q \in \sQ} \left|\mathcal{N}^{(Q)}_{\check{\Phi}}(n,n)\right| \ll n^{\frac{131}{294}+\ep}.
\end{equation}

\subsection{Treatment of $\sK_{\Phi}^{(Q)}(n,n)$}
From the definition \eqref{eq:kphidef} we have 
\begin{multline} \label{Ksum}
\sum_{Q \in \sQ} \left| \sK_{\Phi}^{(Q)}(n,n)\right|
 \ll n^{-\frac{1}{2}} \sum_{Q \in \sQ}\Bigg|  \sum_{\substack{c \leq n \\ c \equiv 0 \pmod{Q} }}  
 c^{-\frac{1}{2}} S(n,n,c,\nu)   J_{\frac{9}{2}} \( \mfrac{4 \pi n}{c} \)  \Bigg| \\
 + n^{-\frac{1}{2}}  \sum_{Q \in \sQ} \Bigg| \sum_{\substack{c>n \\ c \equiv 0 \pmod{Q} }}  
 c^{-\frac{1}{2}} S(n,n,c,\nu)   J_{\frac{9}{2}} \( \mfrac{4 \pi n}{c} \)  \Bigg|.
\end{multline}
Here the situation is simpler since the  $J$-Bessel function has fixed order.
For the first term, we use the expression \eqref{eq:nicelommel} and partial summation
with Proposition~\ref{prop:Iwanbound}, 
together with the fact that $\(x^\frac12 H_\frac{11}2\pfrac x{8\pi n}\)'\ll x^{-\frac12}+x^\frac12 n^{-1}$ for $1\leq x\leq n$.
For the second term, 
\eqref{eq:jderiv} and \cite[(10.14.4)]{nist} give
 $\(J_\frac92\pfrac{4\pi n}x\)'\ll n^\frac92x^\frac{-11}2$ for $x\geq n$.
We conclude that
\begin{equation} \label{Kest}
\sum_{Q \in \sQ}\left|  \sK_{\Phi}^{(Q)}(n,n) \right| \ll n^{\frac{3}{7}+\ep}.
\end{equation}

\subsection{Proof of Theorem~\ref{trade}}
From \eqref{keyineq2}, \eqref{Nest} and \eqref{Kest} we have
\[ \mathcal{L}^{(N)}_{\Hat{\Phi}}(n,n)\ll n^{\frac{131}{294}+\ep}.\]
Until now we have assumed that $N\equiv 0\pmod 8$, but we  may drop this assumption  
using positivity after replacing  $N$ by $2N$ if necessary.
The last inequality gives 
\begin{equation}\label{eq:avgcoeff1}
 n \sum_{j \geq 0} \left|\rho_j(n)\right|^2 \frac{\hat{\Phi}(r_j)}{\ch \pi r_j}\ll n^{\frac{131}{294}+\ep}.
 \end{equation}
From  the asymptotics of the Gamma function \cite[(5.11.9)]{nist}, for $k=\pm \frac12$ we have 
\begin{equation}\label{eq:phistirling}
\hat\Phi(t)\gg t^{k-3}\qquad\text{for $t\geq 1$}.
\end{equation}
By the discussion in Section~\ref{sec:start}, there are positive constants 
$c_1$ and $c_2$ such that for each $r_j\in [0, 1]\cup i(0, 1/4]$ which appears in the sum
\eqref{eq:avgcoeff1} we have
\begin{equation}\label{eq:phiconstant}
c_1\leq \hat\Phi(r_j)\leq c_2.
\end{equation}
So for every $r_j$ which appears in the sum, we have
\begin{equation}\label{eq:philambda}
\hat\Phi(r_j)^{-1}\ll \lambda_j^\frac{3-k}2.
\end{equation}

Suppose that $x \geq 1$. We have $x^{k-3} \Hat{\Phi}^{-1}(r_j) \ll 1$ for $1 \leq r_j \leq x$ by \eqref{eq:phistirling}. The same bound   holds  for   $|r_j| \leq 1$ by \eqref{eq:phiconstant}.
Therefore
 \eqref{eq:avgcoeff1} gives
the average estimate
\begin{equation}\label{eq:avgcoeff}
 n \sum_{|r_j| \leq x} \frac{|\rho_j(n)|^2}{\ch\pi r_j} \ll x^{3-k}n^{\frac{131}{294}+\ep}.
\end{equation} 
Suppose now that $k=\frac12$ as in the statement of Theorem~\ref{trade}.
 When $n>0$, the theorem follows directly from  \eqref{eq:avgcoeff}.
 When $n<0$ it follows from the relationship \eqref{eq:acconj}.

\section{Kuznetsov trace formula in the mixed-sign case} \label{Kuznetsov1}
We give a version of the Kusnetsov trace formula in the mixed sign case which is suitable for our applications.
Suppose that $\phi:[0,\infty) \rightarrow \C$ is four times continuously differentiable and satisfies 
\begin{equation} \label{phicond}
\phi(0)=\phi'(0)=0, \quad \phi^{(j)}(x) \ll_{\ep} x^{-2-\ep} \quad (j=0,\ldots,4) \quad \text{as} \quad x \rightarrow \infty,
\end{equation}
for some $\ep>0$.
Define the transform
\begin{equation}\label{eq:phicheck}
\check{\phi}(r):=\ch \pi r \int_{0}^{\infty} K_{2 i r}(u) \phi(u) \frac{du}{u}
\end{equation}
(we use the notation  $\phi$ instead of $\Phi$ for the rest of the paper to avoid potential confusion with the transform defined in the last section).

Suppose that $N$ is a positive integer and that $\nu$ is a multiplier of weight $\frac12$ for $\Gamma_0(N)$ with the property that no cusp is singular 
with respect to $\nu$.  In this case the Kusnetsov formula has a relatively simple expression since there are no Eisenstein series and
there is no contribution from cusp forms due to the mixed sign of the arguments.
A proof of the following result is  given in Section~4 of  \cite{ahlgren-andersen} in the case when 
$\nu$ is the multiplier on $\SL_2(\mathbb{Z})$ associated to the Dedekind eta function 
(in this case there is no residual spectrum).
The general case follows by this argument with only cosmetic changes.
Blomer \cite{blomer} has proved a version of this formula for the twisted theta-multiplier.

\begin{proposition} \label{Kuz} 
Let $\nu$ be a multiplier of weight $\frac12$ for $\Gamma_0(N)$  such that no cusp is singular with respect to $\nu$.  
Let $\rho_j(n)$ denote the coefficients of an
 orthonormal basis $\{v_j(\tau)\}$ for $\sLT_\frac12(N, \nu)$.
  Suppose that $\phi$ satisfies conditions \eqref{phicond}.  If $m_\nu>0$ and $n_\nu<0$ then
\begin{equation*}
\sum_{\substack{c>0\\c\equiv 0\pmod N}} \frac{S(m,n,c,\nu)}{c} \phi \( \frac{4 \pi \sqrt{m_{\nu} |n_{\nu}|  } }{c} \)=8 \sqrt{i} \sqrt{m_{\nu} |n_{\nu}| } \sum_{r_j} \frac{\overline{\rho_j(m)} \rho_j(n)}{\ch \pi r_j} \check{\phi}(r_j).
\end{equation*}
\end{proposition}

We describe a  family of  test functions.
 Given $a,x>0$, let $T>0$ be a parameter with
\begin{equation}\label{eq:tdef}
T \leq x/3, \quad T \asymp x^{1-\delta} \quad \text{with} \quad 0<\delta<1/2.
\end{equation}

Let $\phi=\phi_{a,x,T}: [0,\infty) \rightarrow [0,1]$ be a smooth function (as in \cite{sarnak-tsimerman} and \cite{ahlgren-andersen}) satisfying 
\begin{enumerate}[\hspace{1.5em}(i)]\setlength\itemsep{.4em}
\item The conditions in \eqref{phicond},
\item $\phi(t)=1$ for $\frac{a}{2x} \leq t \leq \frac{a}{x}$,
\item $\phi(t)=0$ for $t \leq \frac{a}{2x+2T}$ and $t \geq \frac{a}{x-T}$,
\item $\phi^{\prime}(t) \ll \(\frac{a}{x-T}-\frac{a}{x} \)^{-1} \ll  \frac{x^2}{aT}$,
\item $\phi$ and $\phi^{\prime}$ are piecewise monotone on a fixed number of intervals.
\end{enumerate}
We require  bounds for  $\check{\phi}$ which are recorded in \cite[Theorem~6.1]{ahlgren-andersen}.

\begin{proposition}\label{prop:int}
Let a,x,T be as above and let $\phi=\phi_{a,x,T}$. Then
\begin{equation*}
\check{\phi}(r) \ll \begin{cases}
 r^{-\frac{3}{2}} e^{-\frac{r}{2}} & \emph{for} \quad 1 \leq r \leq \frac{a}{8x}, \\
 r^{-1} & \emph{for} \quad \max \(1,\frac{a}{8x} \) \leq r \leq \frac{a}{x},   \\
 \min \(r^{-\frac{3}{2}}, r^{-\frac{5}{2}} \frac{x}{T} \) & \emph{for} \quad r \geq  \max \( \frac{a}{x},1 \). 
\end{cases}
\end{equation*}
\end{proposition}

\section{Proof of Theorem~\ref{thm:powersave} }\label{sec:endgame}

Let $\psi$ be the multiplier defined in \eqref{eq:psidef}.
Theorem~\ref{thm:powersave} will follow from a uniform estimate for sums of
 Kloosterman sums attached to $\psi$.
 We note that many of the terms $x^\ep$ could be changed to $\log x$ factors if necessary, but for simplicity
 we do not keep track of these here.
\begin{theorem} \label{tradesum}
Suppose that  $24n-1$ is positive and squarefree and that $0<\delta<1/2$. For  $X \geq 1$  and $\ep>0$ we have
\begin{equation} \label{demsum}
\sum_{\substack{c \leq X \\ c \equiv 0 \pmod{2} }} \frac{S(0,n,c,\psi)}{c} \ll_{\delta, \ep} \left(n^{\frac{13}{56}+\ep} X^{\frac{3}{4} \delta}+ n^{\frac{143}{588}+\ep}+X^{\frac{1}{2}-\delta} \right) X^{\ep}.
\end{equation} 
\end{theorem}
This  follows in turn from an estimate  over dyadic ranges.

\begin{proposition} \label{tradesumprop}
Suppose that  $24n-1$ is positive and squarefree and that $0<\delta<1/2$. For $x \geq 1$ and $\ep>0$ we have
\begin{equation*}
\sum_{\substack{x \leq c \leq 2x \\ c \equiv 0 \pmod{2} }} \frac{S(0,n,c,\psi)}{c}  \ll_{\delta, \ep} n^{\frac{143}{147}+\ep} x^{-\frac{3}{2}}+n^{\frac{13}{56}} x^{\frac{3}{4} \delta+\ep}+x^{\frac{1}{2}-\delta+\ep}.
\end{equation*}
\end{proposition}
We show that Proposition \ref{tradesumprop} implies Theorem \ref{tradesum}.
Corollary~\ref{weilpsi} below gives a Weil bound for the individual Kloosterman sums, so 
  the initial segment $1 \leq c \leq n^{\frac{143}{294}}$ contributes $O \(n^{\frac{143}{588}+\ep} \)$ to \eqref{demsum}.
 One can then break the interval $n^{\frac{143}{294}} \leq c \leq X$ into $O(\log X)$ dyadic intervals $x \leq c \leq 2x$ with $n^{\frac{143}{294}} \leq x \leq X/2$ and use  Proposition~\ref{tradesumprop} for each interval.

\begin{proof}[Proof of Proposition~\ref{tradesumprop}]
Let  $\{u_j(\tau)\}$ be an orthonormal basis for  $\sLT_{\frac{1}{2}}(2,\psi)$ with Fourier coefficients 
$\rho_j(n)$ and eigenvalues $\lambda_j=\frac{1}{4}+r_j^2$. 
 Recall that $\alpha_\psi=\frac1{24}$, so that $n_\psi=n-\frac1{24}$, and 
define
\begin{equation}\label{eq:adef}
a:=\frac{2 \pi}{\sqrt 6} \sqrt{n_\psi}.
\end{equation}
Let $\phi=\phi_{a,x,T}:[0,\infty) \rightarrow \R$ be a smooth test function with the properties listed in Section~\ref{Kuznetsov1}.

 Using Corollary~\ref{weilpsi} and recalling the definition \eqref{eq:tdef} of $T$ we obtain
\begin{multline} \label{compare}
\left | \sum_{\substack{c>0 \\ c \equiv 0 \pmod{2} }}  \frac{S(0,n,c,\psi)}{c} \phi \left(\frac{a}{c} \right)-
\sum_{\substack{x \leq c \leq 2x \\ c \equiv 0 \pmod{2} }} \frac{S(0,n,c,\psi)}{c}  \right | \\
\leq \sum_{\substack{x-T \leq c \leq x \\ 2x \leq c \leq 2x+2T }} \frac{|S(0,n,c,\psi)|}{c}
\ll_\delta x^{\frac{1}{2}-\delta} \log x.
\end{multline}
We now estimate the smoothed sums.
 Proposition \ref{Kuz} gives
\begin{equation} \label{Kuzapp}
\sum_{\substack{c>0 \\ c \equiv 0 \pmod{2}}}  \frac{S(0,n,c,\psi)}{c} \phi \( \frac{a}{c} \)= 
\frac{4\sqrt{i}}{\sqrt 6} \sqrt{n_\psi} \sum_{r_j} \frac{\overline{\rho_j(0)} \rho_j(n)}{\ch \pi r_j} \check{\phi}(r_j).
\end{equation}

If  $f(\tau)\in \sLT_\frac12(2, \psi)$ has  minimal eigenvalue $\lambda=\frac{3}{16}$,
then by the discussion in Section~\ref{sec:back} $y^\frac14 f(24\tau)\in M_\frac12\(144, \pfrac{12}\bullet\nu_\theta\)$.
By the Serre-Stark basis theorem \cite{serre-stark},
this space is spanned by $\theta(12\tau)=1+\cdots$.
 In view of the 
Fourier expansion \eqref{eq:f_fourier}  this minimal eigenvalue 
 does not occur, so each form $u_j$ appearing in \eqref{Kuzapp} is cuspidal.
 Applying a lift of Sarnak \cite[\S 3]{sarnak-additive}
such a  form implies the existence of  a  non-zero 
form  in the discrete spectrum in $\mathcal{L}_0(72,\mathbf{1})$ with eigenvalue $4 \lambda_j-\frac34>0$
 (the map $\lambda_j\mapsto 4 \lambda_j-\frac34$ corresponds to the map $r_j \mapsto 2r_j$).  
The image of $u_j$ under this lift must therefore be a cusp form (since the residual spectrum occurs only when $\lambda=0$ for congruence groups).   From the numerical data in the L-series and modular forms database \cite{lmfdb} we see that $2r_j>1.1$.  It follows that   $r_j>\frac12$ for all $j$ in the sum \eqref{Kuzapp}.

After this discussion we   break the  sum \eqref{Kuzapp} 
 into three ranges dictated by the behavior of $\check{\phi}$ in Proposition \ref{prop:int}. These are
\begin{equation} \label{specrange}
0.5 < r_j \leq \frac{a}{8x}, \qquad  \max \(1,\frac{a}{8x}\)<r_j<\frac{a}{x}, \qquad \text{and} \quad r_j \geq \max \(1,\frac{a}{x} \). 
\end{equation}

Let $v_j(\tau):=u_j(24 \tau)$.   After   multiplication by a fixed constant, 
Lemma~\ref{lem:level-up} shows that 
 $\{v_j(\tau)\}$ is an orthonormal subset of 
 $\sLT_{\frac{1}{2}}\(144, \left(\frac{12}{\bullet} \right) \nu_{\theta}\)$.
 Denote the Fourier coefficients of $v_j(\tau)$ by $b_j(n)$, so that
\begin{equation} \label{coeffrel}
\rho_j(n)=b_j(24n-1).
\end{equation}

By a straightforward argument using the Weil bound \eqref{eq:weilbound} for the sums $S\(n,n,c, \left( \frac{12}{\bullet} \right) \nu_{\theta}\)$,  \eqref{eq:adassump} is satisfied with the choice $\beta=\frac{1}{2}+\ep$. Using \eqref{coeffrel} and Proposition~\ref{AndDu} we obtain
\begin{equation} \label{smallrj}
\frac{\rho_j(0)}{\sqrt{\ch \pi r_j}}=\frac{b_j(-1)}{\sqrt{\ch \pi r_j}} \ll r_j^\frac54, 
\end{equation}
while Proposition~\ref{avgduke} and  \eqref{coeffrel} give
\begin{equation} \label{smallrj2}
\frac{\rho_j(n)}{\sqrt{\ch \pi r_j }}= \frac{b_j(24n-1)}{\sqrt{\ch \pi r_j  }} \ll n^{-\frac{2}{7}+\ep} r_j^{\frac{9}{4}}.
\end{equation}
Using Proposition \ref{prop:int}, \eqref{smallrj} and \eqref{smallrj2} we obtain
\begin{equation} \label{bound1}
  \sqrt{n_\psi} \sum_{0.5<r_j\leq \frac{a}{8x}}\left| \frac{\overline{\rho_j(0)} \rho_j(n)}{\ch \pi r_j} \check{\phi}(r_j) \right | \ll n^{\frac{3}{14}+\ep} \( \sum_{0.5<r_j \leq 1} |\check{\phi}(r_j)|  +\sum_{1 \leq r_j \leq \frac{a}{8x}} r_j^{2} e^{-\frac{r_j}{2}} \) \ll n^{\frac{3}{14}+\ep}.
\end{equation}
To obtain the last estimate we use the fact that the  first sum on the right contains only finitely many terms, so it is $O(1)$ by  \eqref{eq:phicheck} and \cite[(10.25.3)]{nist}. 
The second sum is also $O(1)$ by Weyl's law for the multiplier $\psi$ \cite[Theorem 2.28,  p.~414]{hejhal-stf1}.

Note that the results in Section~\ref{sec:ell2} apply to the orthonormal collection $\{v_j(\tau)\}$ by positivity.
Using the Cauchy-Schwarz inequality,  Proposition~\ref{AndDu},  Theorem \ref{trade}, and Proposition \ref{prop:int}, we obtain 
\begin{multline} \label{bound2}
\sqrt{n_\psi} \sum_{\frac a{8x} \leq r_j<\frac a{x}} \left | \frac{\overline{\rho_j(0)} \rho_j(n)}{\ch \pi r_j} \check{\phi}(r_j) \right | 
\ll \left( \sum_{\frac a{8x} \leq r_j<\frac a{x}} \frac{|b_j(-1)|^2}{\ch \pi r_j}   \right)^{\frac{1}{2}} \\
\times \left( n_\psi \sum_{\frac a{8x} \leq r_j <\frac a{x}} \frac{|b_j(24n-1)|^2}{\ch \pi r_j} |\check{\phi}(r_j)|^2   \right)^{\frac{1}{2}}
 \ll n^{\frac{131}{588}+\ep} \left( \frac{a}{x} \right)^{\frac{3}{2}} \ll n^{\frac{143}{147}+\ep} x^{-\frac{3}{2}}.
\end{multline}

We record two estimates which are required for the   range $r_j\geq \frac ax$.
 Proposition~\ref{AndDu} gives
\begin{multline} \label{CR1}
\sqrt{n_\psi} \sum_{0 \leq r_j \leq X} \left | \frac{\overline{\rho_j(0)} \rho_j(n)}{\ch \pi r_j} \right | \ll \left( \sum_{0 \leq r_j \leq X} \frac{|b_j(-1)|^2}{\ch\pi r_j} \right)^{\frac{1}{2}} \left( n_\psi \sum_{0 \leq r_j \leq X} \frac{|b_j(24n-1)|^2}{\ch\pi r_j} \right)^{\frac{1}{2}}\\
  \ll\( X^{2}+n^{\frac{1}{4}+\ep} X\)\log X,
\end{multline} 
 while the same argument using Proposition~\ref{AndDu} and Proposition \ref{avgduke} together gives 
\begin{equation} \label{CR2}
\sqrt{n_\psi} \sum_{0 \leq r_j \leq X} \left | \frac{\overline{\rho_j(0)} \rho_j(n)}{\ch \pi r_j} \right |
 \ll \(X^{\frac{7}{2}} n^{\frac{3}{14}+\ep}\)\log X. 
\end{equation}

Let $A \geq \max \(\frac ax,1\)$ and consider the dyadic range $A \leq r_j \leq 2A$. Using the Cauchy--Schwarz inequality,  \eqref{CR1}, \eqref{CR2}
and Proposition \ref{prop:int} we obtain 
\begin{align*}
\sqrt{n_\psi} \sum_{A \leq r_j \leq 2A} \left | \frac{\overline{\rho_j(0)} \rho_j(n)}{\ch \pi r_j} \check{\phi}(r_j) \right | & 
\ll \min \left( A^{-\frac{3}{2}}, A^{-\frac{5}{2}} \frac{x}{T} \right)  \cdot 
\min \left(A^2+n^{\frac{1}{4}+\ep} A, n^{\frac{3}{14}+\ep} A^{\frac{7}{2}}   \right) \log A   \\
& \ll \min \left(A^{\frac{1}{2}}, A^{-\frac{1}{2}} \frac{x}{T}  \right) \cdot 
\min \left(1+n^{\frac{1}{4}+\ep} A^{-1}, n^{\frac{3}{14}+\ep} A^{\frac{3}{2}}  \right)\log A \\
& \ll  n^{\frac{13}{56}+\ep} \min  \( A^{\frac{3}{4}}, A^{-\frac{1}{4}} \frac{x}{T} \)\log A, 
\end{align*}
 where we have used
 \begin{equation*}
 \min(B+C,D) \leq \min(B,D)+\min(C,D) \qquad \text{and} \qquad  \min(B,C) \leq \sqrt{BC}
 \end{equation*}
for positive $B,C$ and $D$. Summing over dyadic integrals yields 
\begin{equation} \label{bound3}
\sqrt{n_\psi} \sum_{r_j \geq  \max (a/x,1)} \left | \frac{\overline{\rho_j(0)} \rho_j(n)}{\ch(\pi r_j)} \check{\phi}(r_j) \right | \ll  n^{\frac{13}{56}+\ep} \( \frac{x}{T} \)^{\frac{3}{4}+\ep} \ll_\delta n^{\frac{13}{56}+\ep} x^{\frac{3}{4} \delta+\ep}.
\end{equation}
Proposition~\ref{tradesumprop} follows from  \eqref{compare}, \eqref{bound1} \eqref{bound2} and \eqref{bound3}.
\end{proof}

We need a simple lemma before proving Theorem \ref{thm:powersave}.  By \cite[(10.30.1)]{nist}, 
 for fixed $\nu,M>0$ and $0 \leq z \leq M$ we have 
\begin{equation} \label{Ibound}
I_{\nu}(z) \ll_{\nu,M} z^{\nu}.
\end{equation}

\begin{lemma} \label{Ichain}
Suppose that $b, \beta>0$. Then for $t/b \geq \beta$ we have 
\begin{equation*}
I_{\frac{1}{2}} \(\mfrac{b}{t} \) \ll_{\beta} \(\mfrac{b}{t} \)^{\frac{1}{2}} \qquad \text{and} \qquad \(I_{\frac{1}{2}} \( \mfrac{b}{t} \) \)^{\prime} \ll_{\beta} b^{\frac{5}{2}} t^{-\frac{7}{2}}+b^{\frac{1}{2}} t^{-\frac{3}{2}}.
\end{equation*}
\end{lemma}
\begin{proof}
The first inequality follows directly from \eqref{Ibound}. From the identity
\begin{equation*}
I_{\frac{1}{2}}^{\prime}(t)=I_{\frac{3}{2}}(t)+\frac{1}{2t} I_{\frac{1}{2}}(t)
\end{equation*}
we obtain
\begin{equation} \label{chain}
\(I_{\frac{1}{2}} \( \mfrac{b}{t} \) \)^{\prime}=-\frac{b}{t^2} \left(I_{\frac{3}{2}} \(\mfrac{b}{t} \)+\frac{t}{2b} I_{\frac{1}{2}} \(\mfrac{b}{t} \)  \right),
\end{equation}
and the second bound follows.
\end{proof}

\begin{proof}[Proof of Theorem \ref{thm:powersave}]
From Section~\ref{sec:fq} we have
\begin{equation} \label{remainderterm}
\bar{R(n,N)}=\frac{2 \pi}{(24n-1)^{\frac{1}{4}}} e \( -\mfrac{1}{8} \) \sum_{\substack{c>2N \\ c \equiv 0 \pmod{2}}} \frac{S(0,n,c,\psi)}{c} I_{\frac{1}{2}} \left( \frac{b}{c} \right),
\end{equation}
where
\begin{equation*}
b:=\frac{\pi \sqrt{24n-1}}{6}.
\end{equation*}
Let 
\begin{equation*}
S(n,X):=\sum_{\substack{c \leq X \\ c \equiv 0 \pmod{2} }} \frac{S(0,n,c,\psi)}{c}.
\end{equation*}

By partial summation we have 
\begin{equation} \label{partialsum}
\sum_{\substack{c>2N \\ c \equiv 0 \pmod{2}}} \frac{S(0,n,c,\psi)}{c} I_{\frac{1}{2}} \left( \mfrac{b}{c} \right)=S(n,2N) I_{\frac{1}{2}} \(\mfrac{b}{2N} \)-\int_{2N}^{\infty} S(n,t) \(I_{\frac{1}{2}} \(\mfrac{b}{t} \) \)^{\prime} dt.
 \end{equation}
Fix $\gamma>0$ and set  $N:=\gamma \sqrt{n}$. Theorem \ref{tradesum} and Lemma \ref{Ichain} imply that 
 \begin{equation*} 
 S(n,2N) I_{\frac{1}{2}} \(\mfrac{b}{2N} \) \ll_{\gamma,\delta} \(n^{\frac{13}{56}+\frac{3}{8} \delta}+n^{\frac{143}{588}}+n^{\frac{1}{4}- \frac{1}{2} \delta} \) n^{\ep}.
 \end{equation*}
 We choose $\delta=\frac{1}{49}$ to obtain 
 \begin{equation} \label{nbound}
 S(n,N) I_{\frac{1}{2}} \(\mfrac{b}{2N} \) \ll_{\gamma} n^{\frac{143}{588}+\ep}.
 \end{equation}
 With this choice of $\delta$, the integral in \eqref{partialsum} also satisfies the bound  \eqref{nbound}.  
 By  \eqref{remainderterm} we have
  \begin{equation*}
 R(n,N) \ll_{\gamma} n^{-\frac{1}{147}+\ep},
 \end{equation*}
 and Theorem~\ref{thm:powersave} is proved.
\end{proof}

\section{Proof of Theorem~\ref{thm:pofn}} \label{sec:pofn}
One can follow  the argument in \cite[\S 9--\S10]{ahlgren-andersen}, using  a modification of the estimate 
(9.11) which deals with intermediate values of the spectral parameter.
 Let $\{u_j(\tau)\}$ be an orthonormal basis for  $\tilde\sL_{\frac{1}{2}}(1,\nu_{\eta})$ with 
  spectral parameters $r_j$ and Fourier coefficients $\rho_j(m)$. Here each  $u_j$  is cuspidal and $r_j>1$ for all $j$ \cite[Corollary~5.3]{ahlgren-andersen}.
Let  $v_j(\tau):=u_j(24 \tau) \in \mathcal{L}_{\frac{1}{2}}\(576, (\frac{12}{\bullet}) \nu_{\theta}\)$;  
after scaling by a fixed constant,  $\{v_j\}$ is an orthonormal set.
If the coefficients are denoted by  $b_j(n)$, then we have
\begin{equation} \label{coeffrel2}
\rho_j(n)=b_j(24n-23).
\end{equation}

Suppose that $24n-23$ is  negative and squarefree.  We have $n_\eta=n-\frac{23}{24}$.
Arguing as in \eqref{bound2} using Theorem~\ref{trade}
we obtain
\begin{multline*}
\sqrt{|n_{\eta}|} \sum_{\frac{a}{8x}<r_j<\frac{a}{x}} \left | \frac{\overline{\rho_j(1)} \rho_j(n)}{\text{ch} \pi r_j } \check{\phi}(r_j)  \right | \ll \left( \sum_{\frac{a}{8x}<r_j<\frac{a}{x}} \frac{|b_j(1)|^2}{\text{ch} \pi r_j} \right)^{\frac{1}{2}}  \\
\times \left( |n_{\eta}| \sum_{\frac{a}{8x}<r_j<\frac{a}{x}} \frac{|b_j(24n-23)|^2}{\text{ch} \pi r_j } |\check{\phi}(r_j)|^2 \right)^{\frac{1}{2}} \ll n^{\frac{143}{147}+\ep} x^{-\frac{3}{2}}.
\end{multline*}
Following the argument  in \cite{ahlgren-andersen} with this estimate  and making the choice $\delta=\frac{1}{49}$
yields Theorem~\ref{thm:pofn}. 

\section{Character sums and the proof of Theorem~\ref{thm:alg}}\label{sec:heegnerconvert}
To translate the error estimates of Theorem~\ref{thm:powersave} to the setting of Theorem~\ref{thm:alg} 
requires a reinterpretation of the Kloosterman sums $S(0, n, 2c, \psi)$ in terms of Weyl-type sums.
To set notation, for $c \in \mathbb{N}$, define
\begin{equation} \label{Fcdef}
F_c(n):= \sum_{\substack{x \pmod{24c} \\ x^2 \equiv 1-24n \pmod{24c}}} \(\frac{-12}{x}  \) e \( \frac{x}{12c} \).
\end{equation}
Then we have 
\begin{proposition}\label{lem:fkmk}
If $c$ is odd then $F_c(n)=0$.  Furthermore we have
\begin{equation}\label{claim0}
F_{2c}(n)=\sqrt{\frac{24}c} e\(-\mfrac18\)\,\, \bar{S(0, n, 2c, \psi)}.
\end{equation}
\end{proposition}
As a consequence we obtain an improved  Weil-type bound for the Kloosterman sums  $S(0, n, 2c, \psi)$.
Computations suggest that this bound  is sharp.
 For example, when $c=15552$ and $n=8278$, the ratio of the left   side to the right is $0.99992\dots$.  
\begin{corollary} \label{weilpsi}
We have 
\begin{equation*}
\left|S(0,n,2c,\psi)\right|\leq 2^{\omega_o(c)} \sqrt{\frac{2c}{(3,c)}},
\end{equation*}
where $\omega_o(c)$ is the number of distinct odd prime divisors of $c$.
\end{corollary}
\begin{proof} 
The congruence $x^2 \equiv 1-24n \pmod{p^\ell}$ has at most four solutions if $p=2$ and at most two solutions otherwise.
 The  claim follows from Proposition~~\ref{lem:fkmk} when $3\mid c$ and Lehmer's bound \eqref{eq:lehmer_bound}, which is stronger when $3\nmid c$.
\end{proof}

Before turning to the proof of Proposition~\ref{lem:fkmk} we 
recall the definition of the Gauss sum
\begin{equation*}
\G(a,b,c):= \sum_{x \pmod{c}} e \( \frac{ax^2+bx}{c} \) \quad c>0, \quad a,b \in \mathbb{Z} \quad \text{and} \quad a \neq 0.
\end{equation*}
For any $d \mid (a,c)$, replacing $x$ with $x+c/d$ gives
\begin{equation} \label{trans}
\G(a,b,c)=e \( \frac{b}{d} \)  \G(a,b,c).
\end{equation}
So $\G(a,b,c)=0$ unless $d \mid b$, in which  case we have 
\begin{equation} \label{factor}
\G(a,b,c)=d \cdot \G \( \frac{a}{d},\frac{b}{d},\frac{c}{d} \).
\end{equation}
If $4 \mid c$ and $(a,c)=1$, then replacing $x$ by $x+c/2$ shows that $\G(a,b,c)=0$ if $b$ is odd. 
These facts together with standard evaluations \cite[\S1.5]{berndt-evans-williams} lead to the formulas
\begin{equation} \label{eval}
\G(a,b,c)=\begin{cases}
e \({\frac{-\overline{a} b^2}{4c}} \)(1+i) \ep_{a}^{-1} \sqrt{c} \( \frac{c}{a} \) & \text{if } b \text{ is even and } 4 \mid c,  \\
e \({\frac{-\overline{4a} b^2}{c}} \) \ep_{c} \sqrt{c} \( \frac{a}{c} \) & \text{if } c \text{ is odd}, \\
0  & \text{if } b \text{ is odd and } 4 \mid c.
\end{cases}
\end{equation}

\begin{proof}[Proof of Proposition~\ref{lem:fkmk}]
For convenience, set
\[D_n:=24n-1.\]
Suppose that $c$ is odd.  If $x$ is a solution to $x^2 \equiv -D_n \pmod{24c}$ then $x+6c$ is also a solution, and 
the corresponding terms have opposite sign.  Therefore $F_c(n)=0$ in this case, and it suffices by Lemma~\ref{lem:fq-kloos} to prove that
\begin{equation} \label{claim}
F_{2c}(n)=\sqrt{\frac{24}c} e\(-\mfrac14\)(-1)^{\lfloor \frac{c+1}{2} \rfloor}A_{2c} \(n-\frac{c(1+(-1)^c)}{4} \).
\end{equation}

By work of Selberg and Whiteman \cite{whiteman} we have 
\begin{equation*}
A_{c}(n):=  \sqrt{\frac{c}{48}} \sum_{ \substack{x \pmod{24c} \\ x^2 \equiv -D_n\pmod {24c} }}  \(\frac{12}{x} \) e \( \frac{x}{12c} \).
\end{equation*}
 For convenience we define
\[
 M_{2c}(n):=A_{2c} \(n-\frac{c \(1+(-1)^c \)}{4} \),
\]
so that
\begin{equation*}
M_{2c}(n)=\sqrt{\frac{c}{24}} \!\!\!\!\!\!\!\!\!
\sum_{\substack{x \pmod{48c} \\ x^2 \equiv -D_n+6 c(1+(-1)^c) \pmod{48c}  }} 
\!\!\!\!\!\!\!\!\!\ \pfrac{12}{x} e \pfrac{x}{12c}.
\end{equation*}
We will make use of the following identities for $x \in \mathbb{Z}$:
\begin{equation}\label{id1}
\begin{aligned}
\pfrac{-12}{x} &=\frac{i}{2 \sqrt{3}} \( e \({-\frac{x}{6}} \)-e \(\frac{x}{6} \)
-e \( \frac{x}{3} \)+e \( {-\frac{x}{3} }\)   \),    \\
\pfrac{12}{x} &=\frac{1}{2 \sqrt{3}} \( e \({\frac{x}{12}} \)-e \(\frac{5x}{12} \)
-e \( {-\frac{5x}{12} }\)+e \( {-\frac{x}{12} }\)   \). 
\end{aligned}
\end{equation}
Since  $F_{2c}$ and $M_{2c}$ each have  period $2c$ in $n$, it suffices to establish the identity 
for the Fourier transforms $\hat F_{2c}$ and $\hat M_{2c}$. We have
\begin{equation*}
\hat F_{2c}(h)=\frac{1}{2c}\sum_{n\pmod{2c}}F_{2c}(n)e\pfrac{hn}{2c}=\frac{1}{2c} \sum_{\substack{x \pmod{48c}}} \(\frac{-12}{x}  \) e \( \frac{x}{24c} \) e \( \frac{h(1-x^2)}{48c} \), 
\end{equation*}
and using \eqref{id1}  we obtain
\begin{multline} \label{fgauss1}
\hat F_{2c}(h)=\frac{i}{4 \sqrt3 c} e \(\frac{h}{48c} \) \Big( \G (-h,2(1-4c),48c )-\G (-h,2(1+4c),48c ) \\
-\G (-h,2(1+8c),48c)+\G (-h,2(1-8c),48c )  \Big).
\end{multline}
Similarly we have  
\begin{equation*}
\hat M_{2c}(h)=\frac{1}{4 \sqrt{6c}} e \pfrac{h \(1+(-1)^c \) }{8} \sum_{x \pmod{48c}} \(\frac{12}{x}  \) e \( \frac{x}{24c} \) e \( \frac{h(1-x^2)}{48c} \),
\end{equation*}
from which 
 \begin{multline} \label{fgauss2}
\hat M_{2c}(h)=\frac1{24 \sqrt{2c}} e \(\frac{h}{48c} + \mfrac{h \(1+(-1)^c \)}{8} \) \Big(\G (-h,2(1+2c),48c ) \\
-\G (-h,2(1+10c),48c)-\G (-h,2(1-10c),48c)+\G (-h,2(1-2c),48c) \Big). 
  \end{multline}
  
Suppose that  $(h,48c)=1$ and define $\hbar$ by  $h\hbar\equiv 1\pmod{48c}$.  Using \eqref{eval} we find that
\[
\hat F_{2c}(h)= \frac{i(1+i)}{\sqrt{c}}  
e \(\frac{h+\hbar}{48c}+\frac{\hbar c}{3}\)   \pfrac{48c}{-h} \ep_{-h}^{-1}  
\( e\pfrac{- \hbar}6   - e\pfrac\hbar6  - e\pfrac\hbar3  + e\pfrac{-\hbar}3 \)
\]
and that 
\begin{multline*}
\hat M_{2c}(h)= \frac{1+i}{\sqrt{24}}
e \(\frac{h+\hbar}{48c}+\frac{\hbar c}{12}\)e\pfrac\hbar4 e\pfrac{h(1+(-1)^c)}8   \pfrac{48c}{-h} \ep_{-h}^{-1}  \\
\cdot \( e\pfrac{- \hbar}6   - e\pfrac\hbar6  - e\pfrac\hbar3  + e\pfrac{-\hbar}3 \).
\end{multline*}
The result follows when  $(h, 48c)=1$  by considering cases $c\pmod 4$.

Note that  each of the Gauss sums $\G(-h, b, 48c)$ appearing in \eqref{fgauss1} and \eqref{fgauss2} has
$(b, 48c)=2$ or $6$.  If $(h, 48c)=2$ or $6$ then $\hat F_{2c}(h)=\hat M_c(h)=0$ by \eqref{eval}.
This leaves   only     the case when $(h, 48c)=3$.
Suppose that this is the case and that $c\equiv 1\pmod 3$.
 Setting $h'=h/3$, we obtain  
\begin{align*} 
\hat F_{2c}(h)&=\frac{3i}{4 \sqrt{3} c} e \pfrac{h}{48c} 
\( \G (-h', \tfrac{2(1-4c)}{3},16c)-\G (-h',\tfrac{2(1+8c)}{3},16c) \)\\
&=\frac{2\sqrt 3 i}{\sqrt c}(1+i)\ep_{-h'}^{-1}\pfrac{16c}{-h'}
e\(\frac h{48c}+\frac{\hbar'}{144c}\)e\pfrac{\hbar'(2c-1)}{18}
\end{align*}
 and 
 \begin{multline*} 
\hat M_{2c}(h)=\frac{1}{8 \sqrt{2 c}}\, e \(\frac{h}{48c}+\mfrac{h \(1+(-1)^c \)}{8} \) 
 \( \G (-h', \tfrac{2(1+2c)}{3},16c)-\G (-h', \tfrac{2(1-10c)}{3},16c) \)\\
 =\frac1{\sqrt2}(1+i)\ep_{-h'}^{-1}\pfrac{16c}{-h'}e\(\frac h{48c}+\frac{\hbar'}{144c}\)
 e\pfrac{h(1+(-1)^c)}8e\pfrac{\hbar'(c+1)}{36}.
\end{multline*}
Comparing these expressions gives \eqref{claim}.  When $c\equiv 2\pmod 3$ the situation is similar and 
we omit the details.
\end{proof}

We turn to the proof of Theorem~\ref{thm:alg}.
For $D>0$ define
\[
	\sQ_{-D, 12} := \left\{ ax^2+bxy+cy^2: b^2-4ac=-D, \ 12\mid a, \ a>0 \right\}.
\] 
Each $Q=[12a,b,c]\in \sQ_{-D, 12}$ has a unique root $\tau_Q\in \H$ given by 
\[\tau_Q=\frac{-b+\sqrt{-D}}{24a}.\]
Matrices $g=\pmatrix \alpha\beta\gamma\delta\in\Gamma_0(12)$ act on these forms by 
\begin{equation*} 
	g Q(x,y) := Q(\delta x-\beta y,-\gamma x+\alpha y).
\end{equation*}
This action preserves $b\pmod {12}$, and 
for   $g\in \Gamma_0(12)$ we have
\begin{equation} \label{eq:root-compatible}
	g\,\tau_Q = \tau_{g Q}.
\end{equation}
For $Q=[12a,b,c]\in \sQ_{-D, 12}$ define 
\[\chi_{-12}(Q)=\pfrac{-12}{b}.\]

Let $\Gamma_\infty\subseteq\Gamma_0(12)$ be the subgroup of translations.
Since $\pmatrix 1t01 [12a, b, \bullet]=[12a, b-24ta, \bullet]$, 
there is a bijection
\begin{equation} \label{eq:sQ-infty-bijection}
	\Gamma_\infty\backslash \sQ_{-D, 12} \longleftrightarrow \left\{(12a,b) : a>0, \  \ \  0\leq b<24a, \ \ b^2\equiv-D\pmod{48a}\right\}.
\end{equation}
Then we have

\begin{proposition}\label{prop:algtoan}  Suppose that $\gamma>0$ and that $n$ is a positive integer.
Then   
\[
\frac{2\pi}{ D_n^{\frac14}}e\(-\mfrac18\)\!\!\! \!\!\!  \sum_{\substack{c\leq\frac{ 2\sqrt{D_n}}{ \gamma}\\c\equiv 0\pmod 2}}
\!\!\! \frac{S(0, n, c, \psi)  }{ c}
I_{\frac12}\pmfrac{\pi \sqrt{D_n}}{6c}
=
\frac{i}{\sqrt{D_n}}\sum_{\substack{Q\in \Gamma_\infty\backslash \sQ_{-D_n, 12}  \\ \im \tau_Q\geq\frac{\gamma}{24} }}
\!\!\! \chi_{-12}(Q)\(e(\tau_Q)-e(\overline\tau_Q)\).
\]

\end{proposition}

\begin{proof}[Proof of Proposition~\ref{prop:algtoan}]
Let $A(n, \gamma)$ denote the sum on the left side of the proposition.
Using Proposition~\ref{lem:fkmk} and the identity
\[I_\frac12(x)=\sqrt{\frac{2}{\pi }} \,\frac{\sh x}{\sqrt{x}},\]
we find that 
\[A(n, \gamma)=\frac 1{2i\sqrt{D_n}}\sum_{a\leq\frac{\sqrt{D_n}}{\gamma}}
\bar{F_{2a}(n)}\(e^\frac{\pi\sqrt{D_n}}{12a}-e^\frac{-\pi\sqrt{D_n}}{12a}\).
\]
Pairing the terms $b$ and $b+24a$  in \eqref{Fcdef} gives
\[
F_{2a}(n)= 2\!\!\!\!\sum_{\substack{b \pmod{24a} \\ b^2 \equiv -D_n \pmod{48a}}} \(\frac{-12}{b}  \) e \( \frac{b}{24a} \),
\]
Therefore
\[A(n, \gamma)=\frac{i}{\sqrt{D_n}}\sum_{\substack{Q\in \Gamma_\infty\backslash \sQ_{-D_n, 12}  \\ \im \tau_Q\geq\frac\gamma{24} }}
\chi_{-12}(Q)\(e(\tau_Q)-e(\overline\tau_Q)\).\]
\end{proof}

\begin{proof}[Proof of Theorem~\ref{thm:alg}]
By \eqref{eq:alphafinal} we have
\[\alpha(n)=A(n, 24\gamma)+R\(n,\, \mfrac{\sqrt{D_n}}{12\gamma}\).\]
Theorem~\ref{thm:alg}  follows from   Proposition~\ref{prop:algtoan}  and Theorem~\ref{thm:powersave}.
\end{proof}

\section{Proof of Theorem~\ref{thm:abs_conv_not}}
Recall the definition \eqref{Fcdef} and fix $n>0$. We will consider the quantities $F_{2p}(n)$ where $p\geq 5$ is a prime
with $\pfrac{1-24n}p=1$.
For such a $p$  let $m_p$ satisfy
\begin{equation}\label{eq:mpdef}
48^2 m_p^2\equiv 1-24n\pmod p.
\end{equation}
 In the sum defining $F_{2p}(n)$ we may take $x=pj\pm 48m_p$,
 where $p^2j^2\equiv 1-24n\pmod{48}$.
 For simplicity, define $\ep_{n,p}\in\{0, 1\}$ by 
 \[1+24\ep_{n,p}\equiv p^2(1-24n)\pmod{48}.\]
Then we have
 \[\begin{aligned}
 F_{2p}(n)&=\sum_{\substack{j^2\equiv 1+24\ep_{n,p}\pmod{48}\\x=pj\pm 48m_p}}
 \pfrac{-12}{x}e\pfrac{x}{24c}\\
 &=2\pfrac{-12}p\cos\pfrac{4\pi m_p}{p}\sum_{j^2\equiv 1+24\ep_{n,p}\pmod{48}}
 \pfrac{-12}{j}e\pfrac{j}{24}.
\end{aligned} \]
Evaluating the last sum, we find that 
\[  F_{2p}(n)=2\sqrt{24}\, i(-1)^n\pfrac{-12}pe\pfrac{p^2-1}{48}\cos\pfrac{4\pi m_p}p.\]

Let $S$ be the set of primes $p\geq 5$ for which there exists an $m_p$ satisfying 
\eqref{eq:mpdef} with $0<\frac{m_p}{p}\leq\frac1{16}$.  For $p\in S$, we  have $F_{2p}(n)\gg 1$.

By a theorem of Duke, Friedlander and Iwaniec \cite{duke-fried-iwan-cong},
we have $\#\{p\in S\ :  p\leq X\}\gg \pi(X)$.
By Proposition~\ref{lem:fkmk} and the fact that $I_\frac12(x) \gg x^\frac12$ as $x\to 0$, 
we find that 
\[\sum_{\substack{ c\leq 2X\\c\equiv 0\pmod 2}}
c^{-1}\left| S(0, n, c, \psi) \right|
I_{\frac12}\pmfrac{\pi \sqrt{24n-1}}{6c}\gg \sum_{\substack{p\leq X\\p\in S}}\frac1p.
\]
Theorem~\ref{thm:abs_conv_not} follows.

\bibliographystyle{amsalpha}
\bibliography{fq}

\newcommand{\noop}[1]{}
\providecommand{\bysame}{\leavevmode\hbox to3em{\hrulefill}\thinspace}
\providecommand{\MR}{\relax\ifhmode\unskip\space\fi MR }
\providecommand{\MRhref}[2]{%
  \href{http://www.ams.org/mathscinet-getitem?mr=#1}{#2}
}
\providecommand{\href}[2]{#2}
\begin{thebibliography}{{LMF}13}

\bibitem[AA18]{ahlgren-andersen}
Scott Ahlgren and Nickolas Andersen, \emph{Kloosterman sums and {M}aass cusp
  forms of half integral weight for the modular group}, International
  Mathematics Research Notices \textbf{2018} (2018), no.~2, 492--570.

\bibitem[AD18]{andersen-duke}
Nickolas Andersen and William Duke, \emph{Modular invariants for real quadratic
  fields and {K}loosterman sums}, Preprint,
  \url{https://arxiv.org/abs/1801.08174} (2018).

\bibitem[Alf14]{alfes}
Claudia Alfes, \emph{Formulas for the coefficients of half-integral weight
  harmonic {M}aa\ss \ forms}, Math. Z. \textbf{277} (2014), no.~3-4, 769--795.
  \MR{3229965}

\bibitem[And66]{andrews}
George~E. Andrews, \emph{On the theorems of {W}atson and {D}ragonette for
  {R}amanujan's mock theta functions}, Amer. J. Math. \textbf{88} (1966),
  454--490. \MR{0200258}

\bibitem[And03]{andrews-survey}
\bysame, \emph{Partitions: at the interface of {$q$}-series and modular forms},
  Ramanujan J. \textbf{7} (2003), no.~1-3, 385--400, Rankin memorial issues.
  \MR{2035813}

\bibitem[BE53]{bateman}
H.~Bateman and A.~Erd{\'e}lyi, \emph{Higher transcendental functions}, Higher
  Transcendental Functions, no. v. 1, McGraw-Hill, 1953.

\bibitem[BEW98]{berndt-evans-williams}
Bruce~C. Berndt, Ronald~J. Evans, and Kenneth~S. Williams, \emph{Gauss and
  {J}acobi sums}, Canadian Mathematical Society Series of Monographs and
  Advanced Texts, John Wiley \& Sons, Inc., New York, 1998. \MR{1625181
  (99d:11092)}

\bibitem[BFOR17]{bfor-book}
Kathrin Bringmann, Amanda Folsom, Ken Ono, and Larry Rolen, \emph{Harmonic
  {M}aass forms and mock modular forms: theory and applications}, American
  Mathematical Society Colloquium Publications, vol.~64, American Mathematical
  Society, Providence, RI, 2017. \MR{3729259}

\bibitem[BH08]{blomer-harcos}
Valentin Blomer and Gergely Harcos, \emph{Hybrid bounds for twisted
  {$L$}-functions}, J. Reine Angew. Math. \textbf{621} (2008), 53--79.
  \MR{2431250}

\bibitem[BHM07]{blomer-harcos-michel}
V.~Blomer, G.~Harcos, and P.~Michel, \emph{A {B}urgess-like subconvex bound for
  twisted {$L$}-functions}, Forum Math. \textbf{19} (2007), no.~1, 61--105,
  Appendix 2 by Z. Mao. \MR{2296066}

\bibitem[Blo08]{blomer}
Valentin Blomer, \emph{Sums of {H}ecke eigenvalues over values of quadratic
  polynomials}, Int. Math. Res. Not. IMRN (2008), no.~16, Art. ID rnn059. 29.
  \MR{2435749}

\bibitem[BM10]{baruch-mao}
Ehud~Moshe Baruch and Zhengyu Mao, \emph{A generalized {K}ohnen-{Z}agier
  formula for {M}aass forms}, J. Lond. Math. Soc. (2) \textbf{82} (2010),
  no.~1, 1--16. \MR{2669637 (2012c:11100)}

\bibitem[BO06]{bringmann-ono}
Kathrin Bringmann and Ken Ono, \emph{The {$f(q)$} mock theta function
  conjecture and partition ranks}, Invent. Math. \textbf{165} (2006), no.~2,
  243--266. \MR{2231957}

\bibitem[DFI95]{duke-fried-iwan-cong}
W.~Duke, J.~B. Friedlander, and H.~Iwaniec, \emph{Equidistribution of roots of
  a quadratic congruence to prime moduli}, Ann. of Math. (2) \textbf{141}
  (1995), no.~2, 423--441. \MR{1324141}

\bibitem[DFI12]{duke-fried-iwan2017}
\bysame, \emph{Weyl sums for quadratic roots}, Int. Math. Res. Not. IMRN
  (2012), no.~11, 2493--2549. \MR{2926988}

\bibitem[{\relax DLMF}]{nist}
\emph{{NIST Digital Library of Mathematical Functions}}, http://dlmf.nist.gov/,
  Release 1.0.18 of 2018-03-27, Online companion to \cite{nist-book}.

\bibitem[Dra52]{dragonette}
Leila~A. Dragonette, \emph{Some asymptotic formulae for the mock theta series
  of {R}amanujan}, Trans. Amer. Math. Soc. \textbf{72} (1952), 474--500.
  \MR{0049927}

\bibitem[Duk88]{duke-half-integral}
W.~Duke, \emph{Hyperbolic distribution problems and half-integral weight
  {M}aass forms}, Invent. Math. \textbf{92} (1988), no.~1, 73--90. \MR{931205
  (89d:11033)}

\bibitem[Duk14]{duke-survey}
\bysame, \emph{Almost a century of answering the question: what is a mock theta
  function?}, Notices Amer. Math. Soc. \textbf{61} (2014), no.~11, 1314--1320.
  \MR{3242661}

\bibitem[FM10]{folsom-masri}
Amanda Folsom and Riad Masri, \emph{Equidistribution of {H}eegner points and
  the partition function}, Math. Ann. \textbf{348} (2010), no.~2, 289--317.
  \MR{2672303 (2011m:11208)}

\bibitem[Hej76]{hejhal-stf1}
Dennis~A. Hejhal, \emph{The {S}elberg trace formula for {${\rm PSL}(2,\mathbb
  R)$}. {V}ol. 1}, Lecture Notes in Mathematics, Vol. 548, Springer-Verlag,
  Berlin-New York, 1976. \MR{0439755 (55 \#12641)}

\bibitem[Hej83]{hejhal-stf2}
\bysame, \emph{The {S}elberg trace formula for {${\rm PSL}(2,\,{\mathbb R})$}.
  {V}ol. 2}, Lecture Notes in Mathematics, Springer-Verlag, Berlin, 1983.
  \MR{711197 (86e:11040)}

\bibitem[Iwa87]{iwaniec-fourier-coefficients}
Henryk Iwaniec, \emph{Fourier coefficients of modular forms of half-integral
  weight}, Invent. Math. \textbf{87} (1987), no.~2, 385--401. \MR{870736
  (88b:11024)}

\bibitem[Kim03]{kim-sarnak}
Henry~H. Kim, \emph{Functoriality for the exterior square of {${\rm GL}_4$} and
  the symmetric fourth of {${\rm GL}_2$}}, J. Amer. Math. Soc. \textbf{16}
  (2003), no.~1, 139--183, With appendix 1 by D. Ramakrishnan and appendix 2 by
  Kim and P. Sarnak. \MR{1937203 (2003k:11083)}

\bibitem[Kno70]{knopp}
Marvin~I. Knopp, \emph{Modular functions in analytic number theory}, Markham
  Publishing Co., Chicago, Ill., 1970. \MR{0265287 (42 \#198)}

\bibitem[Lan00]{landau}
L.~J. Landau, \emph{Bessel functions: monotonicity and bounds}, J. London Math.
  Soc. (2) \textbf{61} (2000), no.~1, 197--215. \MR{1745392}

\bibitem[Leh38]{lehmer-series}
D.~H. Lehmer, \emph{On the series for the partition function}, Trans. Amer.
  Math. Soc. \textbf{43} (1938), no.~2, 271--295. \MR{1501943}

\bibitem[{LMF}13]{lmfdb}
{LMFDB Collaboration}, \emph{The {L}-functions and modular forms database},
  \url{http://www.lmfdb.org}, 2013, [Online; accessed 21 May 2018].

\bibitem[Mas16]{masri-ranks}
Riad Masri, \emph{Singular moduli and the distribution of partition ranks
  modulo 2}, Math. Proc. Cambridge Philos. Soc. \textbf{160} (2016), no.~2,
  209--232. \MR{3458951}

\bibitem[OLBC10]{nist-book}
F.~W.~J. Olver, D.~W. Lozier, R.~F. Boisvert, and C.~W. Clark (eds.),
  \emph{{NIST Handbook of Mathematical Functions}}, Cambridge University Press,
  New York, NY, 2010, Print companion to \cite{nist}.

\bibitem[Pro03]{proskurin-new}
N.~V. Proskurin, \emph{On general {K}loosterman sums}, Zap. Nauchn. Sem.
  S.-Peterburg. Otdel. Mat. Inst. Steklov. (POMI) \textbf{302} (2003), no.~19,
  107--134. \MR{2023036 (2005a:11121)}

\bibitem[Rad36]{rademacher-partition-function}
Hans Rademacher, \emph{On the partition function $p(n)$}, Proc. London Math.
  Soc. \textbf{43} (1936), no.~4, 241--254. \MR{1575213}

\bibitem[Rad43]{rademacher-partition-series}
\bysame, \emph{On the expansion of the partition function in a series}, Ann. of
  Math. (2) \textbf{44} (1943), 416--422. \MR{0008618 (5,35a)}

\bibitem[Rad73]{rademacher-book}
\bysame, \emph{Topics in analytic number theory}, Springer-Verlag, New
  York-Heidelberg, 1973, Edited by E. Grosswald, J. Lehner and M. Newman, Die
  Grundlehren der mathematischen Wissenschaften, Band 169. \MR{0364103 (51
  \#358)}

\bibitem[Sar84]{sarnak-additive}
Peter Sarnak, \emph{Additive number theory and {M}aass forms}, Number theory
  ({N}ew {Y}ork, 1982), Lecture Notes in Math., vol. 1052, Springer, Berlin,
  1984, pp.~286--309. \MR{750670 (86f:11042)}

\bibitem[SS77]{serre-stark}
J.-P. Serre and H.~M. Stark, \emph{Modular forms of weight {$1/2$}}, 27--67.
  Lecture Notes in Math., Vol. 627. \MR{0472707}

\bibitem[ST09]{sarnak-tsimerman}
Peter Sarnak and Jacob Tsimerman, \emph{On {L}innik and {S}elberg's conjecture
  about sums of {K}loosterman sums}, Algebra, arithmetic, and geometry: in
  honor of {Y}u. {I}. {M}anin. {V}ol. {II}, Progr. Math., vol. 270,
  Birkh{\"a}user Boston, Inc., Boston, MA, 2009, pp.~619--635. \MR{2641204
  (2011g:11152)}

\bibitem[Wai17]{waibel}
Fabian Waibel, \emph{Fourier coefficients of half-integral weight cusp forms
  and {W}aring's problem}, The Ramanujan Journal (2017).

\bibitem[Wat36]{watson}
G.~N. Watson, \emph{The final problem : an account of the mock theta
  functions}, J. London Math. Soc. \textbf{11} (1936), no.~1, 55--80.
  \MR{1573993}

\bibitem[Whi56]{whiteman}
Albert~Leon Whiteman, \emph{A sum connected with the series for the partition
  function}, Pacific J. Math. \textbf{6} (1956), 159--176. \MR{0080122
  (18,195b)}

\bibitem[You17]{young}
Matthew~P. Young, \emph{Weyl-type hybrid subconvexity bounds for twisted
  {$L$}-functions and {H}eegner points on shrinking sets}, J. Eur. Math. Soc.
  (JEMS) \textbf{19} (2017), no.~5, 1545--1576. \MR{3635360}

\bibitem[Zag09]{zagier-survey}
Don Zagier, \emph{Ramanujan's mock theta functions and their applications
  (after {Z}wegers and {O}no-{B}ringmann)}, Ast{\'e}risque (2009), no.~326,
  Exp. No. 986, vii--viii, 143--164 (2010), S{\'e}minaire Bourbaki. Vol.
  2007/2008. \MR{2605321}

\bibitem[Zwe02]{zwegers}
S.~P. Zwegers, \emph{Mock theta functions}, Ph.D. thesis, University of
  Utrecht, Utrecht, The Netherlands, 2002.

\end{thebibliography}

\end{document}